\documentclass[11pt,letterpaper]{amsart}
\usepackage{amsmath}
\allowdisplaybreaks
\usepackage{amsthm}
\usepackage{amssymb}
\usepackage[mathscr]{euscript}
\usepackage{mathrsfs}
\usepackage{comment}
\excludecomment{toexclude}
\usepackage{bm}
\usepackage{fancyhdr}
\usepackage{amscd}
\usepackage[all]{xy}
\usepackage{graphicx}
\usepackage{color}
\usepackage{hyperref}
\usepackage{stmaryrd} 
\hypersetup{
  colorlinks,
  citecolor=black,
  filecolor=black,
  linkcolor=black,
  urlcolor=black
}
\usepackage{tikz}
\usetikzlibrary{backgrounds}
\usepackage{slashed}
\usepackage[T1]{fontenc}
\usepackage[utf8]{inputenc}

\usepackage{enumitem,kantlipsum}
\usepackage{soul}

\theoremstyle{plain}
\newtheorem{lemma}{Lemma}[section]
\newtheorem{proposition}[lemma]{Proposition}
\newtheorem{theorem}[lemma]{Theorem}

\newtheorem{Theorem}{Theorem}

\newtheorem{corollary}[lemma]{Corollary}

\newtheorem{Conjecture}{Conjecture}

\theoremstyle{definition}
\newtheorem{definition}[lemma]{Definition}
\newtheorem{Definition}{Definition}

\numberwithin{equation}{section}
\newtheorem{example}[lemma]{Example}

\DeclareMathOperator{\Ric}{Ric}
\DeclareMathOperator{\Div}{div}
\DeclareMathOperator{\Rm}{Rm}

\DeclareMathOperator{\tr}{tr}

\DeclareMathOperator{\ADM}{ADM}

\usepackage[tmargin=1in,bmargin=1in,lmargin=1in,rmargin=1in]{geometry}
\begin{document}

\title{Strict Stability of the ADM Mass and Static Vacuum Extensions}

\author{Zhongshan An}
\address{Institute of Geometry and Physics, University of Science and Technology of China, Pudong New Area, Shanghai 201315, China}
\email{zhshan.an@gmail.com}
\author{Lan-Hsuan Huang}
\address{Department of Mathematics, University of Connecticut, Storrs, CT 06269, USA}
\email{lan-hsuan.huang@uconn.edu}

\begin{abstract}
We introduce a notion of strict stability for the ADM mass on asymptotically flat manifolds with boundary and establish several of its fundamental properties. In particular, we show that a strictly stable static vacuum metric is a strict local ADM mass minimizer, thus proving a local version of  Bartnik's Mass Minimizer Conjecture near Euclidean exterior regions. Our approach also gives a new proof of the existence of static vacuum extensions for nearby Bartnik boundary data on arbitrary exterior regions of Euclidean space, extending \cite{Miao:2003, Anderson:2015, An-Huang:2022}.
\end{abstract}

\maketitle

\section{Introduction}

The stability of Einstein metrics as constrained critical points of the Einstein--Hilbert functional 
\[
\mathcal H(g)=\int_M R_g\,d\mu_g
\]
 is a classical problem. Since the functional is invariant under diffeomorphisms, its Hessian is degenerate along infinitesimal diffeomorphisms, so strict stability must be formulated modulo gauge. On closed manifolds, one typically fixes the volume and works on a slice transverse to the diffeomorphism action, for instance by restricting to transverse-traceless deformations.

The noncompact setting has additional difficulties. The integral  $\mathcal H$ may diverge, and integration by parts may produce nontrivial boundary terms at infinity. After fixing a transverse-traceless gauge, a standard strict-stability condition  is 
\[
\left.D^2\mathcal H\right|_g(h,h) \geq \lambda\|h\|_{L^2(M)}^2
\]
 for compactly supported variations.  However, Euclidean space is not strictly stable in this sense, and such an $L^2$-coercivity condition does not in general imply local minimality of $\mathcal H$. Other notions of stability for complete noncompact manifolds have been studied, notably by Dahl and Kr\"oncke~\cite{Dahl-Kroncke:2024}.

For asymptotically flat manifolds, the ADM mass provides the natural boundary term at infinity. We introduce a notion of strict stability motivated by Choquet-Bruhat and Marsden~\cite{Choquet-Bruhat-Marsden:1976}, who replaced the $L^2$ lower bound by an energy bound involving $\|\nabla h\|_{L^2}^2$. With respect to this notion, Euclidean space is strictly stable, and they proved that the Euclidean metric is a strict local minimizer of the ADM mass, a result commonly referred to as the \emph{local positive mass theorem}.

In this paper, we develop this approach for asymptotically flat static vacuum manifolds with compact boundary. The admissible variations need not be compactly supported and may vary the geometry at the boundary while preserving the Bartnik boundary data. Our notion of strict stability is adapted to these conditions,  and we establish three consequences: local minimality of the ADM mass, perturbative existence and uniqueness of static vacuum extensions, and local Bartnik mass minimizer theorem.

Let $n\ge3$, and let $(M,g)$ be an $n$-dimensional asymptotically flat manifold of rate $q\in(\frac{n-2}{2},n-2)$, possibly having nonempty compact boundary $\Sigma = \partial M$. The \emph{Bartnik boundary data} of $g$ are $(\gamma,H_g)$,  where $\gamma:=g^\intercal$ is the induced metric on $\Sigma$, and $H_g$ is the mean curvature computed by the divergence of the unit normal $\nu_g$ pointing toward the asymptotically flat end. 
A pair $(g,u)$, where $u$ is a nontrivial function satisfying $u\to 1$ at infinity, is called \emph{static vacuum} if
\begin{equation}\label{eq:intro-static}
-u\Ric_g+\nabla_g^2u-(\Delta_gu)g=0.
\end{equation}
Euclidean space and the spatial Schwarzschild manifolds are the basic examples.

For a function $u$ satisfying $u\to1$ at infinity, define the \emph{Regge--Teitelboim functional with the Gibbons--Hawking--York boundary term} by
\[
\mathcal F(g,u)=2(n-1)\omega_{n-1}m_{\ADM}(g)-\int_MuR_g\,d\mu_g+2\int_\Sigma uH_g\,d\sigma_\gamma.
\]
When $u$ is fixed, we write $\mathcal F_u(g):=\mathcal F(g,u)$. On the space of asymptotically flat metrics with fixed induced boundary metric, $\bar g$ is a critical point of $\mathcal F_{\bar u}$ precisely when $(\bar g,\bar u)$ is a static vacuum pair; see Section~\ref{se:rt-variation}.

Fix a static vacuum pair $(\bar g,\bar u)$ and define the scalar-flat constraint with fixed Bartnik data by
\[
\mathscr C_{\bar g}(M):=\bigl\{g\in\mathcal M^{2,\alpha}_{-q}(M):{}R_g=0\text{ in }M,\ g^\intercal=\bar g^\intercal,\ H_g=H_{\bar g}\text{ on }\Sigma\bigr\}.
\]
Its tangent space is
\[
T_{\bar g}\mathscr C_{\bar g}(M)=\bigl\{h\in C^{2,\alpha}_{-q}(M):R'_{\bar g}(h)=0\text{ in }M,\ h^\intercal=0,\ H'_{\bar g}(h)=0\text{ on }\Sigma\bigr\}.
\]
On $\mathscr C_{\bar g}(M)$,
\[
\mathcal F_{\bar u}(g)=2(n-1)\omega_{n-1}m_{\ADM}(g)+2\int_\Sigma\bar uH_{\bar g}\,d\sigma_\gamma,
\]
so $\mathcal F_{\bar u}$ differs from a fixed multiple of the ADM mass by a constant. We call $\bar g$ \emph{stable} if
\[
\left.D^2\mathcal F_{\bar u}\right|_{\bar g}(h,h)\ge0\qquad\text{for every }h\in T_{\bar g}\mathscr C_{\bar g}(M).
\]

Let $\mathscr D(M)$ denote the group of asymptotically Euclidean diffeomorphisms of $M$ that restrict to the identity on $\Sigma$, and let $\mathcal X(M)$ denote its Lie algebra; see Section~\ref{se:af}. Since $\mathcal F_{\bar u}$ restricted to $\mathscr C_{\bar g}(M)$ is invariant under $\mathscr D(M)$,
\[
\left.D^2\mathcal F_{\bar u}\right|_{\bar g}(\mathcal L_X\bar g,\mathcal L_X\bar g)=0\qquad\text{for every }X\in\mathcal X(M).
\]
Thus strict positivity must be formulated modulo diffeomorphisms. For $h\in T_{\bar g} \mathscr C_{\bar g}(M)$, we define the energy norm of the equivalence class $[h]$ modulo gauge transformations by
\[
	\left\| [h]\right\|_{E} := \inf \Big\{ \| \nabla (h + \mathcal L_X \bar g ) \|_{L^2(M)} : X\in \mathcal X(M)\Big\}.	
\]

\begin{Definition}
We say that the static metric $\bar g$ is a \emph{strictly stable} critical point of $\mathcal F_{\bar u}$ on $\mathscr C_{\bar g}(M)$ if there exists $\lambda>0$ such that
\[
\left.D^2\mathcal F_{\bar u}\right|_{\bar g}(h,h)\ge\lambda \left\| [h] \right\|_{E}^2
\]
for every $h\in T_{\bar g}\mathscr C_{\bar g}(M)$. 
\end{Definition}

Our first result shows that strict stability yields a uniform local ADM-mass-minimizing property for all nearby static vacuum metrics.

\begin{Theorem}\label{Th:minimizer}
Let $(M,\bar g,\bar u)$ be an asymptotically flat static vacuum triple, and suppose that $\bar g$ is a strictly stable critical point of $\mathcal F_{\bar u}$ on $\mathscr C_{\bar g}(M)$. Then there exist a neighborhood $\mathcal V\subset\mathcal M^{2,\alpha}_{-q}(M)\times\bigl(1+C^{2,\alpha}_{-q}(M)\bigr)$ of $(\bar g,\bar u)$ in the $C^{2,\alpha}_{-q}$-topology and a constant $\rho>0$ with the following property. For every static vacuum pair $(g,u)\in\mathcal V$ and every $g_1\in\mathscr C_g(M)$ satisfying $\|g_1-g\|_{C^{2,\alpha}_{-q}(M)}<\rho$, 
one has 
\[
m_{\ADM}(g_1)\ge m_{\ADM}(g).
\] 
Equality holds if and only if $g_1$ is diffeomorphic to $g$ by an element  of $\mathscr D(M)$.
\end{Theorem}

In fact, the proof gives a quadratic lower bound for the difference of the ADM masses in the $L^2$-energy distance to the diffeomorphism orbit; see Proposition~\ref{pr:minimizer}. A key analytic ingredient is to identify the tangent spaces of $\mathscr C_g$ as $g$ varies, using a decomposition of symmetric $2$-tensors adapted to scalar curvature deformations and the Bartnik data.  See Theorem~\ref{th:dec}.

These results are closely connected with Bartnik's quasi-local mass program. Let $(\Omega,g_\Omega)$ be a compact Riemannian manifold with smooth connected boundary $\Sigma=\partial\Omega$. Its Bartnik mass is
\[
m_B(\Omega,g_\Omega):=\inf\bigl\{m_{\ADM}(g):(M,g)\text{ is an admissible extension}\bigr\},
\]
where an admissible extension is an asymptotically flat metric $g$ on an exterior manifold $M$ satisfying
\[
R_g\ge0,\qquad R_g\in L^1(M),\qquad g^\intercal=g_\Omega^\intercal,\qquad H_g\le H_{g_\Omega},
\]
together with an appropriate no-horizon condition. For our local results, we fix any diffeomorphism-invariant no-horizon condition that is open under sufficiently small $C^{2,\alpha}_{-q}$-perturbations.

An admissible extension realizing the infimum is called a \emph{Bartnik mass minimizer}. In 1989, Bartnik~\cite{Bartnik:1989} conjectured the existence of mass minimizers. This conjecture was subsequently refined into the following two closely related conjectures; see~\cite{Bartnik:1997,Bartnik:2002}.

\begin{Conjecture}[Bartnik mass minimizer conjecture~{\cite[Conjecture 5]{Bartnik:2002}}]\label{conjecture}
Let $(\Omega,g_\Omega)$ be a compact $3$-dimensional manifold with $R_{g_\Omega}\ge0$ and $H_{g_\Omega}>0$. Then there exists a unique Bartnik mass minimizer $(M,g)$. Moreover, $g$ is static, with a positive static potential $u\to1$ at infinity.
\end{Conjecture}

Under standard hypotheses, a minimizer is static vacuum; see Corvino~\cite{Corvino:2000} and Anderson--Jauregui~\cite{Anderson-Jauregui:2019}. This leads to the extension conjecture.

\begin{Conjecture}[Static vacuum extension conjecture~{\cite[Conjecture 7]{Bartnik:2002}}]\label{conjecture2}
Let $(\Omega,g_\Omega)$ be as above. Then there exists a unique asymptotically flat static vacuum extension $(g,u)$ on $M$, up to diffeomorphism, such that $(\gamma,H_g)=(\gamma_\Omega,H_{g_\Omega})$.
\end{Conjecture}

Additional hypotheses are necessary in both conjectures. Anderson--Jauregui~\cite{Anderson-Jauregui:2019} constructed boundary data that admit no mass-minimizing extension; see also Anderson~\cite{Anderson:2024}. This naturally raises the question of which Bartnik boundary data admit static vacuum extensions and when such extensions are Bartnik mass minimizers.

Perturbative static vacuum extensions for reflection-symmetric boundary data near those of a Euclidean round sphere were constructed by Miao.~\cite{Miao:2003}, and Anderson~\cite{Anderson:2015} later considered the problem without symmetry. In~\cite{An-Huang:2022,An-Huang:2022-JMP}, we developed a general approach that extends the existence theory to star-shaped and generic Euclidean regions. Related results for more general static vacuum backgrounds, including Schwarzschild, have also been studied in~\cite{An-Huang:2024,Ellithy:2026,Alexakis-An-Ellithy-Huang:2026}. The role of strict stability demonstrated in this paper, however, has not previously been used.

A main observation here is that strict stability itself provides a direct mechanism for eliminating the kernel of the linearized Bartnik boundary map. If $(h,v)$ lies in  this kernel, then
\[
\left.D^2\mathcal F_{\bar u}\right|_{\bar g}(h,h)=0.
\]
Strict stability therefore forces $h$ to be pure gauge, eliminating the obstruction to the linearized problem. More generally, we show that strict stability is inherited by exteriors and obtain the existence for any exterior regions. 

\begin{Theorem}\label{Th:extension}
Let $(M,\bar g,\bar u)$ be an asymptotically flat static vacuum triple, possibly with compact boundary. Suppose that $\bar g$ is a strictly stable critical point of $\mathcal F_{\bar u}^M$ on $\mathscr C_{\bar g}(M)$. Let $\hat M\subset M$ be an exterior domain containing the asymptotically flat end, with smooth connected boundary $\hat\Sigma:=\partial\hat M$, and suppose that $\overline{M\setminus\hat M}$ is connected; we allow $\hat M=M$. Assume also that $\bar u>0$ in $\hat M$.

Then there exist constants $\epsilon,C>0$ such that, for every pair of boundary data $(\tau,\phi)$ satisfying
\[
\|\tau-\bar\gamma\|_{C^{2,\alpha}(\hat\Sigma)}+\|\phi-H_{\bar g}\|_{C^{1,\alpha}(\hat\Sigma)}\le\epsilon,
\]
there exists an asymptotically flat static vacuum pair $(g,u)$ on $\hat M$ such that $(g^\intercal,H_g)=(\tau,\phi)$ and
\[
\|(g-\bar g,u-\bar u)\|_{C^{2,\alpha}_{-q}(\hat M)}
\le C\left(\|\tau-\bar\gamma\|_{C^{2,\alpha}(\hat\Sigma)}+\|\phi- H_{\bar g}\|_{C^{1,\alpha}(\hat\Sigma)}\right).
\]
After fixing an elliptic gauge, the solution may be chosen to depend smoothly on $(\tau,\phi)$. Moreover, the solution is geometrically unique in the sense that any two sufficiently nearby solutions are related by an element of $\mathscr D(\hat M)$. Finally, $g$ is a strictly stable critical point of $\mathcal F_u^{\hat M}$ on $\mathscr C_g(\hat M)$.
\end{Theorem}

Since Euclidean space is strictly stable~\cite{Choquet-Bruhat-Marsden:1976} (see also Lemma~\ref{le:stable}), Theorem~\ref{Th:extension} applies with
\[
(M,\bar g,\bar u)=(\mathbb R^n,g_{\mathbb E},1),\qquad \hat M=\mathbb R^n\setminus\Omega.
\]
Thus, for every bounded connected domain $\Omega\subset\mathbb R^n$ with smooth connected boundary, every sufficiently small perturbation of the Euclidean Bartnik boundary data on $\partial\Omega$ admits a nearby, geometrically unique static vacuum extension. This gives a unified perturbative existence result for such  Euclidean exterior regions, extending~\cite{Miao:2003,Anderson:2015,An-Huang:2022,An-Huang:2022-JMP}.

Our final result shows that a strictly stable static vacuum metric is a local minimizer in the larger class relevant to Bartnik mass, in which the scalar curvature is nonnegative and the boundary mean curvature is allowed to decrease.

\begin{Theorem}\label{Th:local}
Let $(M,\bar g,\bar u)$ be an asymptotically flat static vacuum triple with connected compact boundary $\Sigma$. Suppose that $\bar u>0$ and that $\bar g$ is strictly stable. Then there exists a $C^{2,\alpha}_{-q}$-neighborhood $\mathcal V$ of $\bar g$ such that
\[
m_{\ADM}(g)\ge m_{\ADM}(\bar g)
\]
for every $g\in\mathcal V$ satisfying
\[
R_g\ge0,\qquad R_g\in L^1(M),\qquad g^\intercal=\bar g^\intercal,\qquad H_g\le H_{\bar g}.
\]
Equality holds only if $g$ is diffeomorphic to $\bar g$ by an element of $\mathscr D(M)$.

Consequently, if $\bar g$ satisfies a no-horizon condition that is open under sufficiently small perturbations, then $\bar g$ is a strict local Bartnik mass minimizer in the admissible class.
\end{Theorem}

In particular, when applied to Euclidean space, it gives a local affirmative result toward Bartnik's Mass Minimizer Conjecture. More precisely, let $\Omega\subset\mathbb R^n$ be a bounded connected domain with smooth connected boundary. Then there exists $\epsilon>0$ such that any static vacuum extension $(g,u)$ on $\mathbb R^n\setminus\Omega$ satisfying $\|(g-g_{\mathbb E},u-1)\|_{C^{2,\alpha}_{-q}(M)}<\epsilon$ 
is a local Bartnik mass minimizer: $g$ minimizes the ADM mass among admissible extensions in a sufficiently small $C^{2,\alpha}_{-q}$-neighborhood of $g$.

\subsection*{Organization of the paper}

Section~\ref{se:prelim} collects the preliminaries on asymptotically flat manifolds, gives the key equivalent characterization of strict stability, and derives the variational formulas. In Section~\ref{se:min}, assuming the tensor decomposition theorem, we prove Theorem~\ref{Th:minimizer}. Section~\ref{se:de} establishes the tensor decomposition theorem. Finally, we prove Theorem~\ref{Th:extension} and Theorem~\ref{Th:local} in Section~\ref{se:ext}.

\subsection*{Acknowledgments}
Part of this paper was completed while both authors were in residence at the Simons Laufer Mathematical Sciences Institute in Berkeley, California, during the Fall 2024 semester, supported by NSF DMS-1928930. We thank Richard Schoen for helpful discussions. LH was partially supported by NSF DMS-2304966.
\section{Preliminaries}\label{se:prelim}

We introduce the asymptotically flat setting and the gauge projections used to define coercivity modulo diffeomorphisms, then record the variation formulas for $\mathcal F_u$. Along the boundary, the unit normal is chosen to point toward the asymptotically flat end.

\subsection{Asymptotically flat manifolds, weighted spaces, and diffeomorphisms}\label{se:af}

Let $n\ge3$, and let $M$ be a smooth, connected $n$-dimensional manifold with possibly nonempty compact boundary $\Sigma$. We assume that $M$ has one end, identified by a diffeomorphism $\Phi:M\setminus K \to \mathbb R^n\setminus\overline{B_1(0)}$ for some compact set $K\subset M$. Let $\{x^1,\dots,x^n\}$ be the induced coordinates, and extend $r=|x|$ to a smooth positive function on $M$.

The asymptotic chart and a fixed finite atlas on $K$ determine the weighted H\"older spaces $C^{k,\alpha}_{-q}(M)$ and weighted Sobolev spaces $W^{k,p}_{-q}(M)$. We use the standard definitions in \cite[Section~2.1]{Eichmair-Huang-Lee-Schoen:2016}. In particular, a tensor in $C^{k,\alpha}_{-q}(M)$ decays at rate $r^{-q}$, with each derivative contributing one additional power of decay.

Throughout, fix $p\in(1,\infty)$,  $\alpha\in(0,1)$, and $q\in\left(\frac{n-2}{2},n-2\right)$, unless otherwise indicated.  Let $g_{\mathbb E}$ denote the Euclidean metric on the end, smoothly extended to a background metric on $M$. The space of asymptotically flat metrics of rate $q$ is
\[
\mathcal M^{2,\alpha}_{-q}(M) := \left\{ g\text{ a Riemannian metric on }M: g-g_{\mathbb E}\in C^{2,\alpha}_{-q}(M) \right\}.
\]
When $R_g$ is integrable, the ADM mass is
\begin{equation}
m_{\mathrm{ADM}}(g) = \frac{1}{2(n-1)\omega_{n-1}} \lim_{r\to\infty} \int_{S_r} \bigl(\partial_i g_{ij}-\partial_j g_{ii}\bigr) \frac{x^j}{r}\,d\sigma_0,
\end{equation}
where $S_r=\{|x|=r\}$, $d\sigma_0$ is the Euclidean surface measure, and $\omega_{n-1}$ is the area of the Euclidean unit sphere. Repeated indices are summed.

The following inequality allows us to control weighted zeroth-order terms by the gradient energy.

\begin{lemma}[Weighted Poincar\'e inequality; {cf. \cite[Theorem~1.3]{Bartnik:1986}}]\label{le:P}
Let $u\in W^{1,2}_{-q}(M)$. Then $\nabla u\in L^2(M)$ and
\[
\|r^{-1}u\|_{L^2(M)}
\le C\|\nabla u\|_{L^2(M)}.
\]
\end{lemma}

The coordinate freedom at infinity determines the diffeomorphisms relevant to the variational problem. Since transitions between asymptotically flat charts are asymptotic to Euclidean rigid motions \cite[Corollary~3.2]{Bartnik:1986}, we introduce the space of Euclidean Killing fields
\[
\mathcal Z = \operatorname{span} \left\{ \partial_i\ (1\le i\le n),\quad x^i\partial_j-x^j\partial_i\ (1\le i<j\le n) \right\}.
\]
When $M$ is not Euclidean, we regard these as vector fields on $M$ by fixing smooth extensions across a compact region.

Let $\mathscr D(M)$ be the group of $C^{3,\alpha}_{\mathrm{loc}}$ diffeomorphisms that fix $\Sigma$ pointwise and satisfy
\[
\psi(x)=Ox+a+\eta(x) \quad\text{on the asymptotic end}, \qquad O\in SO(n),\quad a\in\mathbb R^n,\quad \eta\in C^{3,\alpha}_{1-q}(M).
\]
The group acts on metrics by pullback. Its Lie algebra is
\begin{equation}\label{eq:X}
\mathcal X(M) = \left\{ X\in C^{3,\alpha}_{\mathrm{loc}}(M): X|_\Sigma=0,\quad X-Z\in C^{3,\alpha}_{1-q}(M) \text{ for some }Z\in\mathcal Z \right\},
\end{equation}
and the corresponding infinitesimal orbit is
\[
T_g\bigl(\mathscr D(M)\cdot g\bigr) = \big\{ \mathcal L_Xg:X\in\mathcal X(M) \big\}.
\]
We also introduce the Sobolev analogue of $\mathcal X(M)$:
\[
\mathcal X^{2,p}(M) = \left\{ X\in W^{2,p}_{\mathrm{loc}}(M): X|_\Sigma=0,\quad X-Z\in W^{2,p}_{1-q}(M) \text{ for some }Z\in\mathcal Z \right\}.
\]
It is used in the definition of elliptic gauge projections in the next section and in Section~\ref{se:ext} to accommodate less regular deformations $h\in W^{1,p}_{-q}(M)$.

\subsection{Elliptic gauge projection}\label{se:gauge}

Since the variational problem is invariant under the action of $\mathscr D(M)$, we work with metric variations modulo infinitesimal diffeomorphisms. Thus, at a background metric $\bar g$, we identify $h_1\sim h_2$ if and only if $h_1 - h_2 = \mathcal L_X \bar g$ for some $X\in \mathcal X(M)$. Denote  the corresponding equivalence class by $[h]$. Rather than working directly with the resulting quotient space, it is more convenient to choose a distinguished representative of each class. We do this by introducing an elliptic gauge projection. Its image gives a space of gauge-fixed variations, which serves as a local slice transverse to the diffeomorphism directions.

\begin{definition}\label{de:gauge}
Let $(M,\bar g)$ be asymptotically flat. For every metric $g$ sufficiently close to $\bar g$ in $C^{2,\alpha}_{-q}(M)$, suppose that $P_g$ is a bounded linear projection on the space of symmetric $(0,2)$-tensors in $W^{1,p}_{-q}(M)$. We call $P_{\bar g}$ an \emph{elliptic gauge projection} if the family satisfies the following properties, with a uniform constant $C>0$:
\begin{enumerate}[label=(\alph*)]
\item For every $h$, there exists $X\in\mathcal X^{2,p}(M)$ such that $P_gh=h+\mathcal L_Xg$. Moreover, if $h\in C^{2,\alpha}_{-q}(M)$, then $X$ can be chosen in $\mathcal X(M)$.
\item\label{it:trivial}
For every $X\in\mathcal X^{2,p}(M)$, $P_g(\mathcal L_Xg)=0$. 
\item
The projection is bounded in the energy norm:
\begin{equation}\label{eq:elliptic}
\|\nabla (P_gh)\|_{L^2(M)} \le C\|\nabla h\|_{L^2(M)}.
\end{equation}
\item The projections depend continuously on the background metric:
\begin{equation}\label{eq:cont}
\|\nabla(P_gh-P_{\bar g}h)\|_{L^2(M)} \le C\|g-\bar g\|_{C^2_{-q}(M)} \|\nabla h\|_{L^2(M)}.
\end{equation}
\item There are neighborhoods $\mathcal V\subset\mathcal M^{2,\alpha}_{-q}(M)$ of $\bar g$ and $\mathcal U\subset\mathcal M^{2,\alpha}_{-q}(M)/\mathscr D(M)$ of $[\bar g]$ such that every $[g]\in\mathcal U$ has a unique representative $\hat g\in\mathcal V$ satisfying $P_{\bar g}(\hat g-\bar g)=\hat g-\bar g$.  The set $\mathcal S_{\bar g}:=\left\{ g\in\mathcal V: P_{\bar g}(g-\bar g)=g-\bar g \right\}$  is the corresponding local slice.
\end{enumerate}
\end{definition}

\begin{example}[Harmonic-orthogonal gauge]\label{ex:ho}
The harmonic-orthogonal gauge combines the Bianchi condition with an orthogonality condition that removes the remaining gauge directions. Define
\[
\beta_gh:=\Div_gh-\frac12d\tr_gh, \qquad \mathcal K_g := \left\{ Z\in\mathcal X(M):\beta_g\mathcal L_Zg=0 \right\}.
\]
A representative $P_gh=h+\mathcal L_Xg$ is selected by
\[
\beta_g(P_gh)=0, \qquad \int_Mr^{-2} \langle P_gh,\mathcal L_Zg\rangle_g\,d\mu_g=0 \quad\text{for every }Z\in\mathcal K_g.
\]
When these conditions can be solved with the uniform estimates and local slice property in Definition~\ref{de:gauge}, they define an elliptic gauge projection.

Let $\mathcal S_g$ be the corresponding local slice. If $M$ is boundaryless and $g$ is Ricci flat, then $T_g\mathcal S_g\cap T_g\mathscr C_g(M)$ consists of
transverse-traceless tensors satisfying the orthogonality condition. At the Euclidean metric on $\mathbb R^n$, $\mathcal K_{g_{\mathbb E}}$ consists of Euclidean Killing fields, so the orthogonality condition is vacuous and the space consists precisely of transverse-traceless tensors.
\end{example}

The next lemma allows us to test strict stability on gauge-fixed representatives rather than on the quotient space.

\begin{lemma}\label{le:equivalent}
Let $(M,\bar g,\bar u)$ be an asymptotically flat static vacuum triple. Then the following are equivalent:
\begin{enumerate}
\item\label{it:1}
$\bar g$ is a strictly stable critical point of $\mathcal F_{\bar u}$ on $\mathscr C_{\bar g}(M)$.

\item\label{it:2}
For some elliptic gauge projection $P_{\bar g}$, there exists $\lambda>0$ such that
\[
\left.D^2\mathcal F_{\bar u}\right|_{\bar g}(h,h) \ge \lambda\|\nabla P_{\bar g}h\|_{L^2(M)}^2
\]
for every $h\in T_{\bar g}\mathscr C_{\bar g}(M)$.

\item For some local slice $\mathcal S_{\bar g}$ associated with an elliptic gauge projection, there exists $\lambda>0$ such that
\[
\left.D^2\mathcal F_{\bar u}\right|_{\bar g}(h,h) \ge \lambda\|\nabla h\|_{L^2(M)}^2
\]
for every $h\in T_{\bar g}\mathscr C_{\bar g}(M) \cap T_{\bar g}\mathcal S_{\bar g}$.
\end{enumerate}
\end{lemma}

\begin{proof}
We prove \eqref{it:1} $\Rightarrow$ \eqref{it:2}; the remaining implications follow from the definitions. For every $X\in\mathcal X(M)$,
\[
\|\nabla P_{\bar g}h\|_{L^2(M)} = \|\nabla P_{\bar g}(h+\mathcal L_X\bar g)\|_{L^2(M)} \le C\|\nabla(h+\mathcal L_X\bar g)\|_{L^2(M)}.
\]
Taking the infimum over $X\in\mathcal X(M)$ and applying \eqref{it:1} yields \eqref{it:2}.
\end{proof}

\subsection{Linearization formulas}

We record the linearizations used in the constraint space and the variations of $\mathcal F_u$. For a smooth family $g(s)$ with $g(0)=g$ and $h=g'(0)$, write
$\mathcal T'_g(h) = \left.D\mathcal T\right|_g(h) := \left.\frac{d}{ds}\right|_{s=0}\mathcal T_{g(s)}$.  Writing $\gamma=g^\intercal$, the volume variations are
\[
(d\mu_g)'(h)=\frac12 (\tr_gh)\,d\mu_g, \qquad (d\sigma_\gamma)'(h) =\frac12 (\tr_\gamma h^\intercal) \,d\sigma_\gamma.
\]

We use the curvature convention $R_{ijk\ell}= g\bigl((\nabla_i\nabla_j-\nabla_j\nabla_i) \partial_k,\partial_\ell\bigr)$,  $\Ric_{ij}=g^{k\ell}R_{kij\ell}$,  and define the Lichnerowicz Laplacian by
\[
(\Delta_Lh)_{ij} = \Delta_gh_{ij} +2R_{ik\ell j}h^{k\ell} -R_i{}^kh_{kj} -R_j{}^kh_{ik}.
\]
We write $\Delta_Lh=\Delta_gh+\mathrm{Rm}*h$. With the Bianchi operator $\beta_g$ defined above, the curvature linearizations are
\begin{align*}
\Ric'_g(h) &= -\frac12\Delta_Lh +\frac12\mathcal L_{(\beta_gh)^\sharp}g,\\
R'_g(h) &= -\Delta_g\tr_gh +\Div_g\Div_gh -\langle\Ric_g,h\rangle_g.
\end{align*}

The adjoint of the scalar curvature linearization is the link between the constraint and the static vacuum equations. Writing $L_gh:=R'_g(h)$, its formal $L^2$-adjoint is
\[
L_g^*u = \nabla_g^2u-(\Delta_gu)g-u\Ric_g.
\]
Thus the static vacuum equation in \eqref{eq:intro-static} is $L_g^*u=0$. Taking its divergence and trace, and using asymptotic flatness and $u\to1$, gives
\[
R_g=0, \qquad \Delta_gu=0, \qquad   \nabla_g^2u=u\Ric_g.
\]

For a fixed smooth function $u$, the Hessian and Laplacian variations are
\begin{align*}
\bigl((\nabla^2)'_g(h)u\bigr)_{ij} &= -\frac12 \left( \nabla_i h_j{}^k +\nabla_j h_i{}^k -\nabla^k h_{ij} \right)\nabla_ku,\\ 
\Delta'_g(h)u &= -\langle h,\nabla_g^2u\rangle_g -\langle\beta_gh,du\rangle_g.
\end{align*}

Along $\Sigma$, let $A_g(X,Y):=g(\nabla_X\nu_g,Y)$ and $H_g:=\tr_\gamma A_g=\Div_g\nu_g$.  Writing $\omega:=h(\nu_g,\cdot)^\intercal$, the mean curvature
linearization is
\[
H'_g(h) = \frac12\nu_g\bigl(\tr_\gamma h^\intercal\bigr) -\Div_\gamma\omega -\frac12h(\nu_g,\nu_g)H_g.
\]

The following identities are obtained by integration by parts.
\begin{lemma}
Let $h$ be a symmetric $(0,2)$-tensor and $v$ a function in $C^2_{-q}(M)$. Then
\begin{align}
\int_MvR'_g(h)\,d\mu_g &= \int_M \left[ \langle\nabla_gv,\nabla_g\tr_gh\rangle_g-(\Div_gh)(\nabla_gv) -v\langle\Ric_g,h\rangle_g \right]\,d\mu_g \notag\\
&\quad+ \int_\Sigma v\left(2H'_g(h)+\langle A_g,h^\intercal\rangle_\gamma +\Div_\gamma\omega \right)\,d\sigma_\gamma \label{eq:ibp1}\\
&= \int_M\langle h,L_g^*v\rangle_g\,d\mu_g+ \int_\Sigma \left[ \langle vA_g-\nu_g(v)\gamma,h^\intercal\rangle_\gamma +2vH'_g(h) \right]\,d\sigma_\gamma.
\label{eq:ibp2}
\end{align}
 \end{lemma}

\subsection{Variations of the Regge--Teitelboim functional} \label{se:rt-variation}

To relate the ADM mass to the static vacuum equations, we use the Regge--Teitelboim functional with the Gibbons--Hawking--York boundary term:
\[
\mathcal F(g,u) = 2(n-1)\omega_{n-1}m_{\mathrm{ADM}}(g) -\int_MuR_g\,d\mu_g +2\int_\Sigma uH_g\,d\sigma_\gamma.
\]
For $g\in\mathcal M^{2,\alpha}_{-q}(M)$ and $u-1\in C^{2,\alpha}_{-q}(M)$, this functional is defined even when $R_g\notin L^1(M)$ by combining the ADM flux
and scalar curvature integral before taking the limit at infinity; see \cite[Section~5]{Huang-Lee:2020}. When $u$ is fixed, we write $\mathcal F_u(g):=\mathcal F(g,u)$.

To express its variations, set
\[
\mathcal E_u(g) := L_g^*u+\frac12uR_gg, \qquad \mathcal B_u(g) := u(A_g-H_g\gamma)-\nu_g(u)\gamma.
\]

\medskip
\noindent\textbf{First variation.}
By \cite[Proposition~3.7]{Anderson-Khuri:2013} and \cite[Lemma~3.1]{Miao:2007},
\[
\left.D\mathcal F_u\right|_g(h) = -\int_M\langle\mathcal E_u(g),h\rangle_g\,d\mu_g -\int_\Sigma \langle\mathcal B_u(g),h^\intercal\rangle_\gamma \,d\sigma_\gamma.
\]
In particular, the boundary term vanishes for variations preserving the induced metric.

\medskip
\noindent\textbf{Second variation.}
We use the Hessian
\[
\left.D^2\mathcal F_u\right|_g(h_1,h_2) := \left.\frac{\partial^2}{\partial t\,\partial s}\right|_{t=s=0} \mathcal F_u(g+sh_1+th_2).
\]
Equivalently, if $g_t = g+ th_2$, then
\begin{align*}
	\left. D^2 \mathcal F_u \right|_g (h_1, h_2) &= -\left. \frac{d}{dt} \right|_{t=0}\int_M \langle \mathcal E_u (g_t), h_1\rangle_{g_t} \, d\mu_{g_t} \\
	&\quad - \left. \frac{d}{dt} \right|_{t=0} \int_\Sigma \left(\langle \mathcal B_u (g_t) , h_1^\intercal \rangle_{\gamma_t} \right)\, d\sigma_{\gamma_t}.
\end{align*}
For the perturbation arguments below, it is useful to denote
a bilinear form:
\begin{equation}\label{eq:Q}
\begin{split}
Q_{g,u}(h_1,h_2) &:= -\int_M \left\langle \left.D\mathcal E_u\right|_g(h_2),h_1 \right\rangle_g\,d\mu_g\\
&\quad+\int_M \left[ 2\langle\mathcal E_u(g)\circ h_2,h_1\rangle_g -\frac12(\tr_gh_2) \langle\mathcal E_u(g),h_1\rangle_g \right]\,d\mu_g. 
\end{split}
\end{equation}
Thus
\[
Q_{g,u}(h_1,h_2) = \left.D^2\mathcal F_u\right|_g(h_1,h_2)
\]
whenever $h_1^\intercal=h_2^\intercal=0$.

At a static background, the terms involving $\mathcal E_u(g)$ vanish. We summarize the resulting observations below.
\begin{lemma}\label{le:sec}
The metric $\bar g$ is a critical point of $\mathcal F_u$ among metrics with induced boundary metric $\bar g^\intercal$ if and only if $(\bar g,u)$ is a static vacuum pair.

At such a pair, let $g(s)$ be a smooth family with $g(0)=\bar g$ and $g(s)^\intercal=\bar g^\intercal$, and set $h=g'(0)$. Then
\begin{align*}
\left.\frac{d^2}{ds^2}\right|_{s=0}\mathcal F_u(g(s)) &= \left.D^2\mathcal F_u\right|_{\bar g}(h,h)\\
&= \int_M \Big\langle u\Ric'_{\bar g}(h)-(\nabla^2)'_{\bar g}(h)u +\left( \Delta'_{\bar g}(h)u-\frac12uR'_{\bar g}(h) \right)\bar g,\, h \Big\rangle_{\bar g}\,d\mu_{\bar g}.
\end{align*}
\end{lemma}

\subsection{The Euclidean space and the scalar-flat constraint}

Euclidean space provides the basic example of strict stability and illustrates the role of the scalar-flat constraint, removing the conformal component of a deformation.

Recall that every symmetric $(0,2)$-tensor $h\in W^{2,p}_{-q}(\mathbb R^n)$ admits the York-type decomposition
\[
h=\varphi g_{\mathbb E}+\mathcal L_Xg_{\mathbb E}+\tau,
\]
where $\varphi\in W^{2,p}_{-q}(\mathbb R^n)$, $X\in W^{3,p}_{1-q}(\mathbb R^n)$, and $\tau\in W^{2,p}_{-q}(\mathbb R^n)$ is transverse-traceless:
\[
\Div_{g_{\mathbb E}}\tau=0, \qquad \tr_{g_{\mathbb E}}\tau=0.
\]
Since $R'_{g_{\mathbb E}}(h) =-(n-1)\Delta_{g_{\mathbb E}}\varphi$,  the linearized scalar-flat constraint and the decay of $\varphi$ imply $\varphi=0$. Thus every admissible deformation has the form
\begin{equation}\label{eq:decom}
h=\mathcal L_Xg_{\mathbb E}+\tau.
\end{equation}

\begin{lemma}[{Choquet-Bruhat and Marsden
\cite[Section~3]{Choquet-Bruhat-Marsden:1976}}]
\label{le:stable}
The Euclidean metric on $\mathbb R^n$ is strictly stable among the scalar-flat
constraint. 
\end{lemma}

\begin{proof}
By Lemma~\ref{le:equivalent},  we choose the transverse-traceless representative $P_{g_{\mathbb E}}h=\tau$ in \eqref{eq:decom}. The transverse-traceless conditions give
\[
\Ric'_{g_{\mathbb E}}(\tau) =-\frac12\Delta_{g_{\mathbb E}}\tau.
\]
By diffeomorphism invariance and Lemma~\ref{le:sec},
\begin{align*}
\left.D^2\mathcal F_1\right|_{g_{\mathbb E}}(h,h)= \left.D^2\mathcal F_1\right|_{g_{\mathbb E}}(\tau,\tau)= -\frac12\int_{\mathbb R^n}\langle\Delta_{g_{\mathbb E}}\tau,\tau\rangle_{g_{\mathbb E}}\,dx= \frac12\int_{\mathbb R^n}|\nabla\tau|^2\,dx.
\end{align*}
\end{proof}

The scalar-flat constraint is essential for this conclusion about $\mathcal F_1$. Away from the constraint, $\mathcal F_1$ is no longer a fixed multiple of the ADM mass. In fact, a conformal deformation can preserve nonnegative scalar curvature while decreasing $\mathcal F_1$, as the following example shows.

\begin{example}
Set $c_n:=4(n-1)/(n-2)$, and choose a nonzero function $f\in C^\infty_c(\mathbb R^n)$ with $f\ge0$. Let $w$ solve $\Delta_{g_{\mathbb E}}w=-f, w(x)\to 0 $ as $|x|\to \infty$.   Then $w>0$. For $s\ge0$ sufficiently small, set $\phi_s:=1+sw,  g_s:=\phi_s^{4/(n-2)}g_{\mathbb E}$.  The conformal scalar curvature formula gives
\[
R_{g_s} = c_ns\phi_s^{-(n+2)/(n-2)}f \ge 0.
\]
On the other hand, the conformal mass and volume formulas yield
\begin{align*}
2(n-1)\omega_{n-1}m_{\ADM}(g_s) &= c_ns\int_{\mathbb R^n}f\,dx,\\
\int_{\mathbb R^n}R_{g_s}\,d\mu_{g_s} &= c_ns\int_{\mathbb R^n}f\,dx +c_ns^2\int_{\mathbb R^n}wf\,dx.
\end{align*}
The linear terms cancel in $\mathcal F_1$, leaving
\[
\mathcal F_1(g_s) = -c_ns^2\int_{\mathbb R^n}wf\,dx = -c_ns^2\int_{\mathbb R^n}|\nabla w|^2\,dx <0 = \mathcal F_1(g_{\mathbb E})
\]
for every sufficiently small $s>0$. Thus $g_{\mathbb E}$ is not a local minimizer of $\mathcal F_1$ in the class $R_g\ge0$, even though it minimizes the ADM mass.

The initial variation $h=g'_0=\frac{4}{n-2}wg_{\mathbb E}$ satisfies $R'_{g_{\mathbb E}}(h)=c_nf\not\equiv0$.  It is therefore not tangent to the scalar-flat constraint,
so Lemma~\ref{le:stable} does not apply. 
\end{example}

\section{Proof of Theorem~\ref{Th:minimizer}}\label{se:min}

The following decomposition theorem for symmetric $(0,2)$-tensors plays a crucial role in the proof of Theorem~\ref{Th:minimizer}. Recall that the formal adjoint of the linearized scalar curvature operator is $L_g^*v:= \nabla_g^2 v - (\Delta_g v) g - v \Ric_g$.
\begin{theorem}\label{th:dec}
Let $(M,g)$ be asymptotically flat. There exists a constant $C>0$ such that every symmetric $(0,2)$-tensor $h\in C^{2,\alpha}_{-q}(M)$ admits a unique decomposition $h = \hat h + k$ where $\hat h\in C^{2,\alpha}_{-q}(M)$ solves $R'_g(\hat h)=0, H'_g(\hat h)=0$ and $k\in C^{2,\alpha}_{-q}(M)$ solves $\Delta_g k = -L_g^* v$ in $M$ and $k|_{\Sigma} =0$ for some scalar function $v\in C^{2,\alpha}_{-q}(M)$ with the energy estimates
\begin{align}
		\| \nabla k \|_{L^2(M)}^2 &=\int_M v R'_g(h)\, d\mu_g -\int_\Sigma 2v H'_g (h)\, d\sigma_\gamma, \label{eq:energyk}\\
	\| \nabla v \|_{L^2(M)} &\le C \| \nabla k\|_{L^2(M)}. \label{eq:energyv}
\end{align}
The constant in \eqref{eq:energyv} may be chosen uniformly when the background metric $g$ ranges over a sufficiently small neighborhood
of a fixed metric in the $C^{2,\alpha}_{-q}(M)$ topology.
\end{theorem}

We defer the proof of Theorem~\ref{th:dec} to Section~\ref{se:de}. Assuming Theorem~\ref{th:dec}, we proceed to prove Theorem~\ref{Th:minimizer}.

The next lemma gives an alternative expression for the bilinear form $Q_{g,u}$ defined in \eqref{eq:Q}, obtained by integrating by parts to remove second derivatives of $k$. The key point is that, when $h^\intercal=0$, the resulting boundary integral involves only $h$, $k$, and their \emph{tangential} derivatives. The proof is a straightforward computation and is omitted.

\begin{lemma}\label{le:Q}
Let $(M,g)$ be asymptotically flat, and let $u$ satisfy $u-1\in C^2_{-q}(M)$. Then for any symmetric $(0,2)$-tensors $h, k\in C^2_{-q}(M)$,
\begin{align*}
Q_{g,u}(h,k)&=\int_M u\ \bigg[ \frac12\langle\nabla h,\nabla k\rangle-\langle\Div_g h,\Div_g k\rangle\\
&\hspace{50pt}+\frac12\left\langle\Div_g h,d(\tr k)\right\rangle+\frac12\left\langle d(\tr h),\Div_g k\right\rangle
\\
&\hspace{50pt}-\frac12\left\langle d(\tr h),d(\tr k)\right\rangle-\frac12\left\langle\Rm_g*k,h\right\rangle
\\
&\hspace{50pt} +\frac12(\tr h)\langle\Ric_g,k\rangle\bigg]\,d\mu_g\\
&\quad+\int_M u_\ell \Big[-k_{ij;i} h_{j\ell }+ h_{ij} k_{i\ell ;j}\\
&\hspace{50pt} +\frac12(h)_{ik}(\tr k)_i +(\tr h)_ik_{i\ell} +\frac12(\tr h) k _{i\ell ;i} \Big]\,d\mu_g
\\
&\quad+\int_M \left(\Delta_g u - \frac12 R_g u \right)\langle h,k\rangle\,d\mu_g \\
&\quad +  \int_M \left( 2\langle \mathcal E_u(g)\circ k , h \rangle_g - \frac{1}{2} (\tr_g k) \langle \mathcal E_u (g), h \rangle_g\right) \, d\mu_g
\\
&\quad +\int_\Sigma\Big[-u\big(\tr_{\gamma}h^\intercal\big)H'(k)+\frac12u\,(\nabla_\nu k)^\intercal\cdot h^\intercal
\\
&\hspace{50pt}+u\left(h_{00}-\frac12\tr h\right)A\cdot k^\intercal-\frac12uH h_{00}k_{00}\\
&\hspace{50pt}-u h _{0a} k_{ab;b}+\frac12u h_{0a}(\tr k)_a+\frac12u k_{0a}(\tr h)_a
\\
&\hspace{50pt} +(\tr h)\ k(\nabla u,\nu)+\frac12(\tr h)k_{0a}u_a \Big]\,d\sigma_\gamma
\end{align*}
where $a,b$ denote tangential indices along $\Sigma$ and $e_0 = \nu$.  

In particular, if $h^\intercal=0$, then we can extend the definition of $Q_{g,u}(h, k)$ to less regular $h, k\in W^{1,2}_{-q}(M)$. For later application, we write the formula as
\begin{align*}
	Q_{g,u}(h, k)&= \int_M\mathcal P_{g, u}(h, \nabla_g h, k, \nabla_g k) \, d\mu_g \\
	&\quad + \int_\Sigma \mathcal S_{g,u} (h, \nabla_{\gamma} h, k, \nabla_{\gamma} k )\, d\sigma_\gamma,
\end{align*}
where, schematically, $*$ denotes contractions with respect to $g$,
\begin{align*}
	\mathcal P_{g, u} (h, \nabla_g h, k, \nabla_g k) &= u \nabla_g h * \nabla_g k + \nabla u * h * \nabla_g k \\
	&\quad + \nabla u * k * \nabla_g h +(\nabla^2u + u\Rm_g)*h*k\\
	\mathcal S_{g,u} (h, \nabla_{\gamma} h, k, \nabla_{\gamma} k )&= u h * \nabla_{\gamma} k + u k * \nabla_{\gamma} h\\
	&\quad + u A * h * k + \nabla_g u * h * k.
\end{align*}

Consequently, if $h^{(j)}, h\in W^{1,2}_{-q}$ satisfy $ \| \nabla (h^{(j)}- h) \|_{L^2(M)}\to 0$ and $(h^{(j)})^\intercal=h^\intercal = 0$, then
\begin{align}\label{eq:Qcont}
	\lim_{j\to \infty} Q_{g,u}(h^{(j)}, h^{(j)}) = Q_{g, u}(h, h).
\end{align}
\end{lemma}

\begin{corollary}\label{co:stable}
Let $(M,\bar g)$ be an asymptotically flat manifold with compact boundary $\Sigma$, and let $\bar u-1\in C^{2,\alpha}_{-q}(M)$. There exists a constant $C>0$ such that, for all $(g,u)$ sufficiently close to $(\bar g,\bar u)$ in the $C^{2,\alpha}_{-q}$-topology and for all symmetric $(0,2)$-tensors $h, k\in W^{1,2}_{-q}(M)$ with $h^\intercal=k^\intercal=0$, we have the following estimates
\begin{align*}
		\left| Q_{\bar g, \bar u} (h, k) \right| &\le C \| \nabla h\|_{L^2(M)}\| \nabla k\|_{L^2(M)} \\
		\left| Q_{g, u} (h, h) - Q_{\bar g, \bar u} (h, h) \right| &\le C\left(\|g-\bar g\|_{C^2_{-q}(M)}+\|u-\bar u\|_{C^2_{-q}(M)}\right)\|\nabla h\|^2_{L^2(M,\bar g)}.
\end{align*}
\end{corollary}
\begin{proof}
Both estimates follow from Lemma~\ref{le:Q}, the weighted Poincar\'e inequality (Lemma~\ref{le:P}), and the trace inequality used to control the boundary terms. It is essential here that the boundary terms are of the form $\mathcal S_{g,u} (h, \nabla_{\gamma} h, k, \nabla_{\gamma} k )$, involving no normal derivatives of $h$ or $k$.

 For example, the trace theorem implies that, for a compact collar $K$ of $\Sigma$,
\[
	\|h\|_{H^{1/2}(\Sigma)}	\le C\|h\|_{H^1(K)}.
\]
By the weighted Poincar\'e inequality,
\[
	\|h\|_{H^1(K)}	\le C\|\nabla h\|_{L^2(M)}.
\]
Therefore, using the duality between $H^{1/2}(\Sigma)$ and $H^{-1/2}(\Sigma)$, we obtain
\begin{align}\label{eq:bdry}
\begin{split}
	\left|\int_\Sigma \langle h,\nabla_{\bar\gamma} k \rangle	\,d\sigma_{\bar\gamma}\right|
	&\le C\|h\|_{H^{1/2}(\Sigma)}	\|\nabla_{\bar\gamma}k \|_{H^{-1/2}(\Sigma)}\\
	&\le C\|h\|_{H^{1/2}(\Sigma)}	\| k \|_{H^{1/2}(\Sigma)}\\
	&\le C\|\nabla h\|_{L^2(M)}	\|\nabla k\|_{L^2(M)}.
\end{split}
\end{align}
\end{proof}

We prove Theorem~\ref{Th:minimizer} in two steps. We first show that  strict stability persists under small static vacuum perturbations, and then use this uniform coercivity to obtain a quadratic lower bound for the ADM mass.

\begin{proposition}\label{pr:str}
Let $(M,\bar g,\bar u)$ be an asymptotically flat static vacuum triple, and suppose that $\bar g$ is a strictly stable critical point of $\mathcal F_{\bar u}$ on $\mathscr C_{\bar g}(M)$. Let $P_{\bar g}$ be an elliptic gauge projection operator. Then there exist a constant $\lambda_0>0$ and a neighborhood $\mathcal V\subset \mathcal M^{2,\alpha}_{-q}(M)\times\bigl(1+C^{2,\alpha}_{-q}(M)\bigr)$ of $(\bar g,\bar u)$ such that every static vacuum pair $(g,u)\in\mathcal V$ satisfies
\[
\left.D^2\mathcal F_u\right|_g(h,h) \ge \lambda_0\|\nabla (P_gh)\|_{L^2(M)}^2 \qquad\text{for every }h\in T_g\mathscr C_g(M).
\]
\end{proposition}

\begin{proof}
The main point is to compare the constraint tangent spaces at $g$ and $\bar g$. We use Theorem~\ref{th:dec} to correct a tangent variation at $g$ into one at $\bar g$, with a small error in the energy norm.

Set
\[
\delta:=\|g-\bar g\|_{C^{2,\alpha}_{-q}(M)}+\|u-\bar u\|_{C^{2,\alpha}_{-q}(M)}.
\]
 For $\delta$ sufficiently small, the $L^2(M, g)$-norms are uniformly equivalent. We therefore suppress the dependence on $g$. By diffeomorphism invariance, it suffices to consider $h\in T_g\mathscr C_g(M)$ satisfying $P_gh=h$.
 
 Apply Theorem~\ref{th:dec} to obtain $(k,v)$ such that $h-k\in T_{\bar g}\mathscr C_{\bar g}(M)$. Since
$R'_g(h)=0, H'_g(h)=0$, comparing the first-order expressions in \eqref{eq:ibp1} at $\bar g$ and $g$, and using  \eqref{eq:energyk}, \eqref{eq:bdry},
gives
\begin{align*}
\|\nabla k\|_{L^2(M)}^2&= \int_M vR'_{\bar g}(h)\,d\mu_{\bar g} -2\int_\Sigma vH'_{\bar g}(h)\,d\sigma_{\bar\gamma}\\
&= \int_M\Big[ \langle \nabla_{\bar g} v , \nabla_{\bar g} \tr_{\bar g} h \rangle_{\bar g} - \Div_{\bar g} h (\nabla_{\bar g} v) - v \langle \Ric_{\bar g} , h \rangle_{\bar g} \Big] \, d\mu_{\bar g}\\
&\quad + \int_\Sigma v \Div_{\bar \gamma} (h(\nu_{\bar g}, \cdot)^\intercal) \, d\sigma_{\bar \gamma}\\
&\le C\delta\|\nabla v\|_{L^2(M)} \|\nabla h\|_{L^2(M)}.
\end{align*}
Together with \eqref{eq:energyv}, this yields
\begin{equation}\label{eq:nearby-k}
\|\nabla k\|_{L^2(M)} \le C\delta\|\nabla h\|_{L^2(M)}.
\end{equation}

By Corollary~\ref{co:stable} and \eqref{eq:nearby-k},
\[
\bigl|Q_{g,u}(h,h)-Q_{\bar g,\bar u}(h-k,h-k)\bigr| \le C\delta\|\nabla h\|_{L^2(M)}^2.
\]
By \eqref{eq:cont}, \eqref{eq:nearby-k}, and $P_gh=h$,
\begin{align*}
&\|\nabla(P_{\bar g}(h-k)-h)\|_{L^2(M)} \\
&\le \| \nabla (P_{\bar g} - P_g) h \|_{L^2(M)}  + C \| \nabla k \|_{L^2(M)} \le C\delta\|\nabla h\|_{L^2(M)}.
\end{align*}
Corollary~\ref{co:stable} and strict stability at $\bar g$ therefore give
\begin{align*}
\left.D^2\mathcal F_u\right|_g(h,h) &\ge Q_{\bar g,\bar u}(h-k,h-k) -C\delta\|\nabla h\|_{L^2(M)}^2\\
&\ge \lambda\|\nabla P_{\bar g}(h-k)\|_{L^2(M)}^2 -C\delta\|\nabla h\|_{L^2(M)}^2\\
&\ge (\lambda-C\delta)\|\nabla h\|_{L^2(M)}^2.
\end{align*}
Shrinking $\mathcal V$ gives the assertion.

\end{proof}

\begin{proposition}\label{pr:minimizer}
Let $(M,\bar g,\bar u)$ be as in Proposition~\ref{pr:str}. Then there exist a neighborhood $\mathcal V\subset \mathcal M^{2,\alpha}_{-q}(M)\times\bigl(1+C^{2,\alpha}_{-q}(M)\bigr)$ of $(\bar g,\bar u)$ and constants $\rho,\epsilon>0$ with the following property. For every static vacuum pair $(g,u)\in\mathcal V$, let $\mathcal S_g$ be the local slice associated with $P_g$, and set
\[
\mathcal U_g(\rho):= \left\{[g_1]\in\mathscr C_g(M)/\mathscr D(M): \begin{array}{l} \text{the slice representative }\hat g_1\in\mathcal S_g\text{ satisfies}\\
\|\hat g_1-g\|_{C^{2,\alpha}_{-q}(M)}<\rho \end{array} \right\}.
\]
Then
\[
m_{\ADM}(g_1)-m_{\ADM}(g)\ge \epsilon\,d_g([g_1],[g])^2
\]
for every $[g_1]\in\mathcal U_g(\rho)$, where
\[
d_g([g_1],[g]):= \inf_{\widetilde g_1\in[g_1]} \|\nabla (\widetilde g_1-g)\|_{L^2(M,g)}.
\]
Equality holds only if $[g_1]=[g]$.
\end{proposition}

\begin{proof}
By Proposition~\ref{pr:str}, after shrinking $\mathcal V$ there is a constant $\lambda_0>0$, independent of $(g,u)\in\mathcal V$, such that
\begin{equation}\label{eq:uni}
\left.D^2\mathcal F_u\right|_g(\xi,\xi) \ge\lambda_0\|\nabla P_g\xi\|_{L^2(M,g)}^2 \qquad\text{for every }\xi\in T_g\mathscr C_g(M).
\end{equation}
By shrinking $\mathcal V$ further, the constants in Theorem~\ref{th:dec}, Corollary~\ref{co:stable}, the weighted Poincar\'e and trace inequalities, and the energy estimates for the gauge projections may be chosen uniformly.  

Let $[g_1]\in\mathcal U_g(\rho)$, let $\hat g_1$ be its slice representative, and set $h:=\hat g_1-g$. Then $P_gh=h$ and $h^\intercal=0$. Consider the straight-line path $g_s:=g+sh$, $0\le s\le1$, and set
\[
\delta:=\|h\|_{C^{2,\alpha}_{-q}(M)}.
\]
Although $h$ need not lie in $T_g\mathscr C_g(M)$, the endpoint constraints imply that its failure to satisfy the linearized constraints is quadratic:  By Taylor's formula,
\begin{align} \label{eq:taylor}
\begin{split}
	0&= R_{\hat g_1} - R_{g} = R'_{g}(h) +  \int_0^1 (1-s) D^2 R|_{g_{s}} (h, h) \, ds, \\
	0&= H_{\hat g_1} - H_{g} = H'_{g}(h) + \int_0^1 (1-s) D^2 H|_{g_{s}} (h, h) \, ds.
\end{split}
\end{align}

Apply Theorem~\ref{th:dec} at $g$ to obtain $(k,v)$ such that
$h-k\in T_g\mathscr C_g(M)$. From \eqref{eq:energyk} and \eqref{eq:taylor}, we have 
\begin{align}\label{eq:k-error}
\begin{split}
\| \nabla k \|_{L^2(M)}^2 &= \int_M v R'_{g}(h)\, d\mu_{g}-\int_\Sigma 2v H'_{g} (h)\, d\sigma_{\gamma}\\
&= -\int_0^1 (1-s)  \int_M v D^2  R|_{g_s} (h, h)  \,d\mu_{g}\, ds\\
&\quad  + \int_0^1 (1-s) \int_\Sigma 2vD^2 H|_{g_s}(h, h) \,  d\sigma_{\gamma} \, ds.
\end{split}
\end{align}
Explicitly, the integrands satisfy the quadratic bounds:
\begin{align*}
	&|D^2 R|_{g_s} (h, h) |\le C ( |h ||\nabla^2 h|  + |\nabla h|^2  + r^{-2-q} |h|^2  ), \\
	&|D^2 H|_{g_s} (h, h)| \le C (|h| |\nabla h|  + |h|^2 ).
\end{align*}
Using the $C^{2,\alpha}_{-q}$-smallness of $h$ to bound one factor in each quadratic term, and then applying the weighted Poincar\'e and trace
inequalities give, uniformly in $s$, we have 
\[
\begin{split}
&\int_M |v|\, \bigl|\left.D^2R\right|_{g_s}(h,h)\bigr|\,d\mu_g +\int_\Sigma |v|\, \bigl|\left.D^2H\right|_{g_s}(h,h)\bigr|\,d\sigma_\gamma\\
&\hspace{25mm}\le C\delta\|\nabla v\|_{L^2(M)} \|\nabla h\|_{L^2(M)}.
\end{split}
\]
Thus \eqref{eq:k-error} and \eqref{eq:energyv} imply
\[
\|\nabla k\|_{L^2(M)}^2 \le C\delta\|\nabla v\|_{L^2(M)} \|\nabla h\|_{L^2(M)} \le C\delta\|\nabla k\|_{L^2(M)} \|\nabla h\|_{L^2(M)},
\]
and hence
\begin{equation}\label{eq:uniform-k}
\|\nabla k\|_{L^2(M)} \le C\delta\|\nabla h\|_{L^2(M)}.
\end{equation}

By Corollary~\ref{co:stable}, \eqref{eq:uniform-k}, and
\eqref{eq:uni}, together with $P_gh=h$ and the uniform boundedness
of $P_g$, 
\begin{align*}
Q_{g_s,u}(h,h) &\ge Q_{g,u}(h-k,h-k) -C\delta\|\nabla h\|_{L^2(M)}^2\\
&\ge (\lambda_0-C\delta)\|\nabla h\|_{L^2(M)}^2.
\end{align*}
Since the first variation of $\mathcal F_u$ at $g$ vanishes by Lemma~\ref{le:sec}, choosing $\rho$ sufficiently small, we obtain by Taylor's formula, 
\[
\mathcal F_u(\hat g_1)-\mathcal F_u(g) = \int_0^1(1-s)Q_{g_s,u}(h,h)\,ds \ge \frac{\lambda_0}{4}\|\nabla h\|_{L^2(M)}^2.
\]

After converting the functional difference to the mass difference, we have 
\[
m_{\ADM}(g_1)-m_{\ADM}(g) \ge \epsilon \|\nabla h\|_{L^2(M)}^2 \ge \epsilon\,d_g([g_1],[g])^2
\]
for some $\epsilon>0$. If equality holds, then $\nabla h=0$. Since $h$ decays at infinity, $h=0$, and hence $[g_1]=[g]$.
\end{proof}

\section{Scalar curvature deformation and tensor divergence equation}\label{se:de}

In this section, we prove Theorem~\ref{th:dec}. The aim is to solve the linearized scalar curvature and boundary constraints while controlling the correction in the Dirichlet energy norm. We complement these underdetermined constraint equations with an auxiliary equation involving a function $v$, together with suitable boundary conditions. A similar idea has been implemented in \cite[Appendix A]{Huang-Jang:2022}. The main analytical issue, however, is to control the energy norms of $k, v$, and to this end we introduce entirely new complementing equations.

In Section~\ref{se:scalar}, we establish solvability of the coupled boundary value problem. In Section~\ref{se:div}, we prove $\|\nabla v\|_{L^2(M)} \le C\|\nabla k\|_{L^2(M)}$. The proof relies on solving a tensor divergence equation with zero boundary values. Combining Proposition~\ref{pr:k} and Proposition~\ref{pr:en} then proves Theorem~\ref{th:dec}.

Throughout, $(M,g)$ is an $n$-dimensional asymptotically flat manifold of rate $q$ with compact boundary $\Sigma$, and recall $L_g^*v:= \nabla_g^2 v - (\Delta_g v) g - v \Ric_g$.

\subsection{Prescribing scalar curvature and Bartnik boundary data}\label{se:scalar}

To complement the underdetermined geometric constraints, we impose the auxiliary equation
\[
\Delta_g k+L_g^*v=0
\]
and strengthen the boundary condition $k^\intercal=0$ to $k|_\Sigma=0$. These conditions are chosen both to give an elliptic boundary value problem and to eliminate the boundary term when deriving the energy estimates.
\begin{proposition}\label{pr:k}
Given $f\in C^{0,\alpha}_{-q-2}(M), \phi\in C^{1,\alpha}(\Sigma)$, there exists a unique pair consisting of a symmetric $(0,2)$-tensor $k\in C^{2,\alpha}_{-q}(M)$ and a function $v\in C^{2,\alpha}_{-q}(M)$ satisfying
\begin{align*}
	&R_g'(k) = f, \qquad \Delta_g k + L_g^* v=0 \qquad \mbox{ in } M\\
	&k=0, \qquad H_g'(k) = \phi \qquad \mbox{ on } \Sigma. 
\end{align*}
The energy of $k$ satisfies
\begin{equation}\label{eq:energyk2}
\|\nabla k\|_{L^2(M)}^2 = \int_Mvf\,d\mu_g-2\int_\Sigma v\phi\,d\sigma_\gamma.
\end{equation}
\end{proposition}
The identity \eqref{eq:energyk2} follows directly from \eqref{eq:ibp2}. The existence statement follows from the next two lemmas, which show that the coupled system is an elliptic boundary value problem and an isomorphism on the corresponding weighted spaces.

\begin{lemma}\label{le:elliptic}
Consider the map from a symmetric $(0,2)$-tensor $k$ and a function $v$ given by
\begin{align}\label{eq:L}
(k, v)\mapsto F(k, v) := \left( R'_g(k),\, \Delta_g k+L_g^*v,\, k|_\Sigma,\, H'_g(k)\right),
\end{align}
where the first two components are defined on $M$ and the last two on $\Sigma$. Then the system is elliptic and the boundary conditions satisfy the complementing condition.
\end{lemma}

\begin{proof}
For $\xi\ne0$, vanishing of the principal symbol of the second
interior equation gives
\[
k_{ij} = v\left(\delta_{ij} -\frac{\xi_i\xi_j}{|\xi|^2}\right).
\]
Substituting into the symbol of the scalar equation yields $(n-1)|\xi|^2v=0$. Thus $v=0$ and $k=0$. Hence the principal symbol is an isomorphism, and the interior system is elliptic.

For the complementing boundary condition, freeze the coefficients at a boundary point and choose flat coordinates $(y^1,\dots,y^{n-1},t)$ centered at that point, where \(t\ge0\) is the normal variable, the boundary is \(\{t=0\}\), and $g_{ij} = \delta_{ij}$ at the origin. For $\zeta = (\zeta_1, \dots, \zeta_{n-1})$ with $|\zeta|=1$, consider decaying solutions of the form
\[
v(y,t)=e^{i \zeta \cdot y}V(t), \qquad k_{ij}(y,t)=e^{i \zeta \cdot y}K_{ij}(t).
\]
The principal interior equations become
\begin{align*}
 R_0' (K):=-\Delta_0 (\tr K) + \Div_0 \Div_0 K &=0 \\
\Delta_0 K_{ij}+\partial_i\partial_jV-(\Delta_0 V)\delta_{ij}&=0, 
\end{align*}
where, after tangential Fourier transform,
\[
\partial_A\mapsto i\zeta_A, \qquad \partial_n\mapsto \frac{d}{d t} \qquad \Delta_0 \mapsto \frac{d^2}{dt^2} - 1, 
\]
where  $A=1, \dots, n-1$ denotes the tangential indices. The boundary conditions are
\[
	K_{ij}(0)=0, \quad \sum_A K_{AA}'(0)=0.
\]
Using the Hermitian pairing for complex tensors, integration by parts in $t$ therefore gives
\[ 
0 =\int_0^\infty \overline V\,R'_0(K)\,dt=\int_0^\infty \langle K,L_0^*V\rangle_{\mathbb C}\,dt=\int_0^\infty\bigl(|K'|^2+|K|^2\bigr)\,dt.
\]
Here the boundary terms vanish by the boundary conditions and decay. Hence $K=0$. The normal-normal component of $L_0^*V=0$ is $V=0$, so the decaying solution is trivial.
\end{proof}

\begin{lemma}\label{le:iso}
Let $0<\delta<n-2$. Let $F$ be the elliptic boundary operator defined in Lemma~\ref{le:elliptic}, acting on pairs $(k,v)$ whose components lie in $C^{2,\alpha}_{-\delta}$, with the natural weighted H\"older target spaces for its two interior and two boundary components. Then $F$ is an isomorphism.
\end{lemma}

\begin{proof}
By Lemma~\ref{le:elliptic} and the weighted Fredholm theory for elliptic boundary problems on asymptotically flat manifolds, $F$ is Fredholm for every $0<\delta<n-2$. It therefore suffices to prove that its kernel and cokernel are trivial.

Define the interior operator
\[
P_g(k,v):=\bigl(R'_g(k),\Delta_gk+L_g^*v\bigr),
\]
and let $P_0$ denote its Euclidean constant-coefficient model:
\[
P_0(k,v) = \left( -\Delta\tr k+\partial_i\partial_jk_{ij}, \Delta k_{ij}+\partial_i\partial_jv-(\Delta v)\delta_{ij} \right).
\]
The principal-symbol calculation in Lemma~\ref{le:elliptic} shows that $P_0$ is a square, second-order constant-coefficient elliptic system. By the weighted constant-coefficient isomorphism theorem \cite[Theorem~3]{Lockhart-McOwen:1983}, together with standard scaled Schauder estimates,
\[
P_0:C^{2,\alpha}_{-\mu}(\mathbb R^n) \longrightarrow C^{0,\alpha}_{-\mu-2}(\mathbb R^n)
\]
is an isomorphism for every $0<\mu<n-2$.

We show that the kernel of $F$ is trivial. Let $(k,v)\in C^{2,\alpha}_{-\delta'}(M)$ satisfy $F(k,v)=0$. Choose a smooth cutoff function $\chi$ such that $\chi=0$ for $r\le R$ and $\chi=1$ for $r\ge 2R$. Then $\chi(k,v)$ may be regarded as a pair defined on $\mathbb R^n$, extended by zero across the interior region. Since $g-g_{\mathbb E}\in C^{2,\alpha}_{-q}$, the difference of the operators has the schematic form 
\[
P_0-P_g=a(x)\partial^2+b(x)\partial+c(x),
\]
 where $a\in C^{0,\alpha}_{-q}, b\in C^{0,\alpha}_{-q-1}, c\in C^{0,\alpha}_{-q-2}$. Therefore, $P_0\bigl(\chi(k,v)\bigr)\in C^{0,\alpha}_{-2-\delta'-q}(\mathbb R^n)$.  The isomorphism property of $P_0$ then implies that $k,v\in C^{2,\alpha}_{-\delta}$ for any $0<\delta<\min\{\delta'+q,n-2\}$.  Iterating this argument finitely many times yields $k,v\in C^{2,\alpha}_{-\delta}$ for every $0<\delta<n-2$.

This improved decay implies that all boundary terms at infinity vanish. Hence,
\begin{align*}
	0&=\int_M vR'_g(k)\,d\mu_g \\
	&=\int_M \langle k,L_g^*v\rangle_g\,d\mu_g
		+2\int_\Sigma vH'_g(k)\,d\sigma_g\\
	&=\int_M |\nabla k|^2\,d\mu_g.
\end{align*}
It follows that $k\equiv0$, and hence $L_g^*v=0$. A nontrivial static potential on an asymptotically flat end is asymptotic to an affine function. Since $v$ decays, it follows that $v\equiv0$.

We show that the cokernel is trivial. Let $(f, \tau, \kappa,\phi)$ be a cokernel element: for every pair $(k,v)$ compactly supported up to the boundary,
\begin{align*}
0 &=\int_M \left\langle \bigl(R'_g(k),\Delta_gk+L_g^*v\bigr),(f,\tau) \right\rangle_g\,d\mu_g\\
&\quad+ \int_\Sigma \left\langle\bigl(k,H'_g(k)\bigr), (\kappa,\phi)\right\rangle_\gamma \,d\sigma_\gamma.
\end{align*}
Integration by parts gives
\begin{align*}
0 &=\int_M \left\langle \bigl(R'_g(\tau),\Delta_g\tau+L_g^*f\bigr),(v,k) \right\rangle_g\,d\mu_g\\
&\quad+ \int_\Sigma \left\langle \bigl(fA-\nu(f)\gamma+\nabla_\nu\tau,2f\bigr), \bigl(k,H'_g(k)\bigr) \right\rangle_\gamma\,d\sigma_\gamma\\
&\quad- \int_\Sigma \left\langle \bigl(vA-\nu(v)\gamma+\nabla_\nu k,2v\bigr), \bigl(\tau,H'_g(\tau)\bigr) \right\rangle_\gamma\,d\sigma_\gamma\\
&\quad+ \int_\Sigma \left\langle(\kappa,\phi),\bigl(k,H'_g(k)\bigr)\right\rangle_\gamma \,d\sigma_\gamma.
\end{align*}
Varying $(k,v)$ gives
\[
R'_g(\tau)=0, \qquad \Delta_g\tau+L_g^*f=0 \quad\text{in }M,
\]
\[
\tau=0, \qquad H'_g(\tau)=0 \quad\text{on }\Sigma,
\]
and $\kappa=-fA+\nu(f)\gamma-\nabla_\nu\tau$, $\phi=-2f$. 

By weighted elliptic duality, a cokernel element has the complementary decay $(\tau,f)=O\bigl(r^{-(n-2-\delta)}\bigr)$. Since $0<n-2-\delta<n-2$, the decay-improvement argument above applies to $(\tau,f)$. The preceding kernel argument then gives $\tau=f=0$, and the boundary equations imply $\kappa=\phi=0$. Thus the cokernel is trivial. Since $F$ is Fredholm, it is an isomorphism.
\end{proof}

A direct corollary of independent interest is the following nonlinear version, which gives an effective version of a local submersion result of Anderson and Jauregui~\cite[Proposition 2.4]{Anderson-Jauregui:2019}.
\begin{corollary}
Fix a background asymptotically  flat manifold $(M, \bar g)$. The map from $(g, v) \in \mathcal M^{2, \alpha}_{-q}(M) \times C^{2, \alpha}_{-q}(M)$ defined by 
\[
(g, v) \mapsto \Big(R_g, \Delta_{\bar g} g+(\nabla^2_gv-(\Delta_gv)g-v\Ric_g ), g|_{\Sigma}, H_g\Big)
\]
with values in the corresponding target weighted H\"older spaces is a local diffeomorphism at $(\bar g, 0)$.
\end{corollary}
\begin{proof}
The differential of this map at $(\bar g, 0)$ is precisely the operator $F$ in Lemma~\ref{le:elliptic}. The result therefore follows from Lemma~\ref{le:iso} and the inverse function theorem.
\end{proof}

\subsection{The tensor divergence equation with prescribed boundary value}\label{se:div}

The remaining task is to estimate the auxiliary function $v$ from the equation $\Delta_gk+L_g^*v=0$. For this purpose, we seek a $(0,2)$-tensor $T$, not necessarily symmetric, satisfying
\[
\Div_gT=\nabla v,\qquad T|_\Sigma=0, \qquad \|\nabla T\|_{L^2(M)} \le C\|\nabla v\|_{L^2(M)},
\]
where the divergence operator is defined as $(\Div_g T)_i=\nabla^j T_{ij}$. The key property is the zero boundary condition. (Although $T=vg$ solves the divergence equation, it does not satisfy the boundary condition in general.) This divergence equation will be used to  convert the energy of $v$ into a pairing with its Hessian, while the zero boundary value eliminates the boundary terms. 

Divergence equations with the zero boundary condition have been studied extensively, for example, in the work of Bogovski\u{\i}  on the divergence equation for vector fields~\cite{Bogovskii:1979}. In the Euclidean case, our result follows from this result by solving the equation componentwise. In our setting, we work on general Riemannian domains and use elliptic theory to establish existence.

Let us set up the stage. Given a one-form $\eta$ on $U$ and a $(0,2)$-tensor $B$ prescribed on the boundary, we seek a tensor $T$ solving $\Div_g T = \eta$ and $T|_{\partial U}=B$. We address this problem in three steps:
\begin{enumerate}
\item Let $U$ be a bounded domain with smooth boundary, and assume that $\eta$ satisfies the compatibility condition
\[
\eta\in\mathcal P(U)^\perp,\qquad\mathcal P(U):=\left\{\omega\text{ a one-form on }U:\nabla\omega=0\right\},
\]
where the orthogonal complement is taken with respect to the $L^2$ inner product.
\item On an asymptotically flat manifold $M$, remove the compatibility condition on $\eta$.
\item Solve the problem with inhomogeneous boundary data.
\end{enumerate}

We begin with the first step.

\begin{lemma}\label{le:div}
Let $(U,g)$ be a Riemannian manifold with compact closure and smooth boundary. Let $\mathcal P(U)$ denote the space of parallel one-forms on $U$. There exists a constant $C>0$ such that the following holds. For every one-form $\eta\in L^2(U)$ satisfying
\begin{align}\label{eq:int}
\int_U \langle \eta,\omega\rangle\,d\mu_g=0 \qquad \text{for every }\omega\in\mathcal P(U),
\end{align}
there exists a $(0,2)$-tensor $T\in H^1(U)$ such that
\begin{align*}
\Div_g T=\eta, \qquad T|_{\partial U}=0,
\end{align*}
and
\begin{align}\label{eq:energy1}
\|\nabla T\|_{L^2(U)}\leq C\|\eta\|_{L^2(U)}.
\end{align}
\end{lemma}
\begin{proof}
We seek a solution of the form $T=\rho\nabla\tau$, where $\tau$ is a one-form and $\rho$ is a function vanishing at the boundary. This leads to a weighted variational problem for $\tau$. The main point is to ensure that $T$ has zero boundary trace.

Choose $\rho\in C^2(\overline U)$ positive in $U$ and equal
to the boundary distance in a collar of $\partial U$. We use the
weighted norms
\[
\|u\|_{L^2_\rho(U)}^2 :=\int_U\rho|u|^2\,d\mu_g, \qquad \|u\|_{H^1_\rho(U)}^2 :=\int_U\rho\bigl(|u|^2+|\nabla u|^2\bigr)\,d\mu_g.
\]
By  the weighted Sobolev inequality (see, for example, \cite[Theorem 1.5]{Necas:1962}),
\[
\|u\|_{L^2(U)} \le C\|u\|_{H^1_\rho(U)}.
\]
A compactness argument on the $L^2$-orthogonal complement of $\mathcal P(U)$ then gives
\begin{equation}\label{eq:P}
\|u\|_{L^2(U)} \le C\|\nabla u\|_{L^2_\rho(U)} \qquad \text{for }u\in H^1_\rho(U)\cap\mathcal P(U)^\perp.
\end{equation}

For $\tau,\zeta\in H^1_\rho(U)\cap \mathcal P(U)^\perp$, define the bilinear form
\begin{align*}
B(\tau,\zeta):=\int_U \rho\langle\nabla\tau,\nabla\zeta\rangle\,d\mu_g.
\end{align*}
 By \eqref{eq:P} and the Lax--Milgram theorem, there is a unique $\tau\in H^1_\rho(U)\cap\mathcal P(U)^\perp$ satisfying
\begin{equation}\label{eq:bi}
\int_U\rho\langle\nabla\tau,\nabla\zeta\rangle\,d\mu_g = -\int_U\langle\eta,\zeta\rangle\,d\mu_g \qquad \text{for every }\zeta\in H^1_\rho(U),
\end{equation}
where we used the compatibility condition \eqref{eq:int} to extend the variational identity from $\mathcal P(U)^\perp$ to all test one-forms. Testing with $\zeta=\tau$ and using \eqref{eq:P} gives 
\begin{equation}\label{eq:weighted-tau}
\|\tau\|_{L^2(U)} +\|\nabla\tau\|_{L^2_\rho(U)} \le C\|\eta\|_{L^2(U)}.
\end{equation}
Define $T=\rho\nabla\tau$, with the convention $T_{ij}=\rho\nabla_j\tau_i$. Then $T\in L^2(U)$ and $\Div_gT=\eta$ weakly.

The weighted estimate~\eqref{eq:weighted-tau} does not yet give the boundary regularity for $\nabla \tau$.  Write $U_\delta=\{\rho<\delta\}$. For a fixed sufficiently small $\delta>0$, choose a $C^1$ vector field $X$ equal to $\nabla\rho$ on $U_\delta$ and vanishing outside $U_{2\delta}$. Taking $\zeta=\nabla_X\tau$ in \eqref{eq:bi} and integrating by parts gives
\begin{align*}
\int_U\langle\eta,\nabla_X\tau\rangle\,d\mu_g&=\int_U\left[\tfrac12\Div_g(\rho X)|\nabla\tau|^2-\rho(\nabla_iX^k)\langle\nabla^i\tau,\nabla_k\tau\rangle\right]\,d\mu_g \\
&\quad+\int_U \rho X^k\left\langle\nabla^i\tau,(\mathrm{Rm})^\ell{}_{jik}\tau_\ell\right\rangle\,d\mu_g.
\end{align*}
On $U_\delta$, we have $\Div_g(\rho X)=\Div_g(\rho\nabla\rho)=1+\rho\Delta\rho$.  After decreasing $\delta$ if necessary, it follows that
\begin{align*}
\tfrac12\Div_g(\rho\nabla\rho)|\nabla\tau|^2-\rho(\nabla^2\rho)(\nabla\tau,\nabla\tau)\geq\tfrac14|\nabla\tau|^2 \quad \text{on } U_\delta.
\end{align*}
 Using Cauchy--Schwarz and Young's inequality on the other integrals, and then absorbing the resulting gradient terms, we obtain
\begin{align*}
\|\nabla\tau\|_{L^2(U_\delta)}^2\leq C\|\eta\|_{L^2(U)}^2+C\|\nabla\tau\|_{L^2_\rho(U)}^2.
\end{align*}
The terms in $U\setminus U_\delta$ are controlled by \eqref{eq:weighted-tau}, since $\rho\geq\delta$ there. Together with \eqref{eq:P} and  \eqref{eq:weighted-tau}, this yields
\[
\|\tau\|_{H^1(U)}\leq C\|\eta\|_{L^2(U)}.
\]

It remains to prove that $T\in H^1_0(U)$. Rather than estimate all second derivatives of $\tau$, we use the Gaffney estimate: For a  $(0,2)$-tensor $S$ in $H^1_0(U)$,  we write $S_{i[j;k]}=S_{ij;k}-S_{ik;j}$. Integration by parts and the Ricci identity give
\begin{equation}\label{eq:gaffney}
\|S\|_{H^1(U)}
\le C\left(
\|S_{i[j;k]}\|_{L^2(U)}
+\|\Div_g S\|_{L^2(U)}
+\|S\|_{L^2(U)}
\right).
\end{equation}
Choose a smooth cutoff $\chi_\delta$ that vanishes on $U_\delta$, equals one outside $U_{2\delta}$, and satisfies $\rho|\nabla\chi_\delta|\le C$. Set
\[
T_\delta:=\chi_\delta T.
\]
Interior elliptic regularity gives $\tau\in H^2_{\mathrm{loc}}(U)$, so $T_\delta\in H^1_0(U)$. The Ricci identity, $\Div_gT=\eta$, and the cutoff bound imply
\[
\begin{split}
&\|(T_\delta)_{i[j;k]}\|_{L^2(U)} +\|\Div_gT_\delta\|_{L^2(U)} +\|T_\delta\|_{L^2(U)}\\
&\le C\bigl(\|\tau\|_{H^1(U)}+\|\eta\|_{L^2(U)}\bigr) \le C\|\eta\|_{L^2(U)}.
\end{split}
\]
Applying \eqref{eq:gaffney}, we obtain
\[
\|T_\delta\|_{H^1(U)} \le C\|\eta\|_{L^2(U)}
\]
uniformly in $\delta$.

The dominated convergence theorem gives
\begin{align*}
T_\delta \rightarrow T \qquad \text{strongly in }L^2(U).
\end{align*}
After passing to a subsequence, the uniform $H^1$ estimate also gives
\begin{align*}
T_\delta\rightharpoonup T \qquad \text{weakly in }H^1(U).
\end{align*}
Since $H^1_0(U)$ is weakly closed and $T_\delta \in H^1_0(U)$, it follows that $T\in H^1_0(U)$. In particular, $T$ has zero boundary trace. Then  \eqref{eq:energy1} follows by weak lower semicontinuity
\begin{align*}
\|\nabla T\|_{L^2(U)}\leq\liminf_{\delta\to0}\|\nabla T_\delta \|_{L^2(U)}\leq C\|\eta\|_{L^2(U)}.
\end{align*}
\end{proof}

\begin{lemma}\label{le:div2}
Let $(M,g)$ be an asymptotically flat manifold with compact boundary $\Sigma$. Then there exist a sufficiently large smooth compact domain $K\subset M$ and a constant $C>0$ such that, for every one-form $\eta\in L^2(M)$ supported in $K$, there exists a $(0,2)$-tensor $T\in H^1(M)$ satisfying
\[
  \Div_g T=\eta,   \qquad   T|_{\Sigma}=0,
\]
and
\[
  T=O(r^{1-n})  \qquad\text{as } r\to\infty.
\]
Moreover,
\[
  \|\nabla T\|_{L^2(M)}  \le C\|\eta\|_{L^2(M)}.
\]
\end{lemma}
\begin{proof}
To apply Lemma~\ref{le:div} on a compact domain, the source term must be orthogonal to parallel one-forms. We remove this obstruction by constructing tensors whose fluxes at infinity account for the projection of $\eta$ on the space of parallel one-forms.

Let $\{\omega_{(1)},\dots,\omega_{(N)}\}$ be a pointwise orthonormal basis of the space $\mathcal P(M)$ of global parallel one-forms.
Choose $r_0$ sufficiently large that the asymptotically flat coordinates are defined outside $B_{r_0}$, and let $\nu$ denote the outward unit normal to the coordinate spheres. For each $i$, set
\[
S_{(i)}=\widehat\omega_{(i)}\otimes\nu^\flat,
\]
where $\widehat\omega_{(i)}$ solves, along the integral curves of $\nu$,
\[
\nabla_\nu\widehat\omega_{(i)} +(\Div_g\nu)\widehat\omega_{(i)}=0, \qquad \widehat\omega_{(i)}|_{\partial B_{r_0}} =\frac{\omega_{(i)}}{|\partial B_{r_0}|}.
\]
This choice gives $\Div_gS_{(i)}=0$ on the end with the asymptotics
\[
S_{(i)}=O(r^{1-n}), \qquad \nabla S_{(i)}=O(r^{-n}).
\]
Since each $\omega_{(j)}$ is parallel, the flux is independent of $r$ and, by the initial normalization,
\begin{equation}\label{eq:flux}
\int_{\partial B_r} S_{(i)}(\omega_{(j)}^\sharp,\nu)\,d\sigma =\delta_{ij}, \qquad r\ge r_0.
\end{equation}

Extend $S_{(i)}$ to a tensor $\widetilde S_{(i)}$ on $M$ that agrees with $S_{(i)}$ outside $B_{r_0}$ and vanishes near $\Sigma$. Then $\widetilde S_{(i)}\in H^1(M)$ and $\Div_g\widetilde S_{(i)}$ is compactly supported. The divergence theorem and \eqref{eq:flux} give
\begin{equation}\label{eq:pairing-S}
\int_M \langle\Div_g\widetilde S_{(i)},\omega_{(j)}\rangle\,d\mu_g =\delta_{ij}.
\end{equation}
Fix a smooth compact domain $K\supset B_{r_0}$ sufficiently large so that every parallel one-form on $K$ extends to $M$. Given $\eta$ supported in $K$, define
\[
c_i:=\int_M\langle\eta,\omega_{(i)}\rangle\,d\mu_g, \qquad \widetilde\eta :=\eta-\sum_{i=1}^N c_i\Div_g\widetilde S_{(i)}.
\]
The corrected source $\widetilde\eta$ is supported in $K$ and orthogonal to $\mathcal P(K)$ by \eqref{eq:pairing-S}. Lemma~\ref{le:div} therefore gives
$\widetilde T_K\in H^1_0(K)$ with $\Div_g\widetilde T_K=\widetilde\eta$. Its zero extension $\widetilde T$ to $M$ satisfies
\begin{equation}\label{eq:Ttilde-est}
\|\nabla\widetilde T\|_{L^2(M)}
\le C\|\widetilde\eta\|_{L^2(M)}.
\end{equation}

Set
\[
T:=\widetilde T+\sum_{i=1}^N c_i\widetilde S_{(i)}.
\]
Then $T\in H^1(M)$, $\Div_gT=\eta$, and $T|_\Sigma=0$. Outside $K$, only the flux tensors remain, so $T=O(r^{1-n})$.
Finally, Cauchy--Schwarz gives $|c|\le C\|\eta\|_{L^2(M)}$. Since the tensors $\widetilde S_{(i)}$ are fixed, \eqref{eq:Ttilde-est} yields
\[
\|\nabla T\|_{L^2(M)} \le C\bigl(\|\widetilde\eta\|_{L^2(M)}+|c|\bigr) \le C\|\eta\|_{L^2(M)}.
\]
\end{proof}

\begin{lemma}\label{le:div-b}
There exist a smooth compact domain $K\subset M$ and a constant $C>0$ such that the following holds. Let $\eta \in L^2(M)$ be supported in $K$ and let $B\in H^{1/2}$ be an ambient $(0,2)$-tensor along $\Sigma$. Then there exists a $(0,2)$-tensor $T\in H^1(M)$ satisfying $T=O(r^{1-n})$ and
\begin{align*}
	 \Div_g T &= \eta, \qquad T|_{\Sigma}=B
\end{align*}
with
\[
	\|\nabla T\|_{L^2(M)}\le C\left(\|\eta\|_{L^2(M)}+\|B\|_{H^{1/2}(\Sigma)}\right).
\]
\end{lemma}
\begin{proof}
By the trace extension theorem, there is an extension $\widetilde B$ on $M$, supported in a fixed collar of $\Sigma$, with $\|\widetilde B\|_{H^1(M)}\le C\|B\|_{H^{1/2}(\Sigma)}$. After enlarging $K$ to contain this collar, Lemma~\ref{le:div2} gives $\widetilde T\in H^1(M)$ with $\widetilde T=O(r^{1-n})$ such that
\begin{align*}
	\Div_g \widetilde T &=\eta - \Div_g\widetilde B, \quad \text{and} \quad \widetilde T|_{\Sigma}=0.
\end{align*}
Then $T:=\widetilde T+\widetilde B$ is the desired solution.

\end{proof}

\begin{corollary}\label{co:div2}
There exists a constant $C>0$ such that, for every $v\in C^{2,\alpha}_{-q}(M)$, there exists a $(0,2)$-tensor $T$ such that $T-vg\in H^1(M)$, $T-vg=O(r^{1-n})$, and
\begin{align*}
	\Div_g T &= \nabla v \quad \text{and} \quad	T|_{\Sigma} =0.
\end{align*}
Moreover,
\[
	\| \nabla T \|_{L^2(M)} \le C \| \nabla v \|_{L^2(M)}.
\]
\end{corollary}
\begin{proof}
By Lemma~\ref{le:div-b}, we can solve for $\widetilde T$ with
\[
	\Div_g \widetilde T = 0 \quad \text{and} \quad \widetilde T|_{\Sigma} = -vg
\]
with the estimate
\[
	\| \nabla \widetilde T \|_{L^2(M)} \le C \| v \|_{H^{\frac12}(\Sigma)} \le C\| v \|_{H^1(K)} \le C \| \nabla v \|_{L^2 (M)}
\]
where the last inequality follows from the weighted Poincar\'e inequality and the decay of $v$. Then $T:=\widetilde T+vg$ is the desired solution.
\end{proof}

The preceding corollary provides a test tensor that recovers the energy of $v$ from the auxiliary equation without introducing boundary terms. We now use it to prove the remaining estimate in Theorem~\ref{th:dec}.

\begin{proposition}\label{pr:en}
There exists a constant $C>0$ such that, whenever $k,v\in C^{2,\alpha}_{-q}(M)$ with $q>\frac{n-2}{2}$ satisfy
\[
	\Delta_gk+L_g^*v=0,
\]
one has
\begin{align}\label{eq:est}
	\| \nabla v \|_{L^2(M)} \le C \| \nabla k \|_{L^2(M)}.
\end{align}
\end{proposition}
\begin{proof}
By Corollary~\ref{co:div2}, there is a $(0,2)$-tensor $T$ satisfying
\[
\Div_gT=\nabla v, \qquad T|_\Sigma=0, \qquad \|\nabla T\|_{L^2(M)} \le C\|\nabla v\|_{L^2(M)}.
\]
To match the principal part of $L_g^*$, set
\[
\widehat T:=T-\frac{1}{n-1}(\tr_gT)g.
\]
Then
\[
\langle T,\nabla^2v\rangle = \langle\widehat T,\nabla^2v-(\Delta_gv)g\rangle.
\]
Using $\Div_gT=\nabla v$ and $L_g^*v=-\Delta_gk$, integration by parts gives
\begin{align*}
\|\nabla v\|_{L^2(M)}^2 &=-\int_M\langle T,\nabla^2v\rangle\,d\mu_g\\
&=-\int_M\langle\nabla\widehat T,\nabla k\rangle\,d\mu_g -\int_Mv\langle\widehat T,\Ric_g\rangle\,d\mu_g.
\end{align*}
These integrations are justified by the regularity supplied by Corollary~\ref{co:div2}, the zero boundary trace, and decay at infinity.

The first term has the desired bound $C\|\nabla v\|_{L^2(M)}\|\nabla k\|_{L^2(M)}$. The curvature term is small outside a compact set: for any $\epsilon>0$, asymptotic flatness and the weighted Poincar\'e inequality give a sufficiently large compact set $K$ such that
\[
\left| \int_Mv\langle\widehat T,\Ric_g\rangle\,d\mu_g\right| \le C\|v\|_{L^2(K)}\|\nabla v\|_{L^2(M)} +\epsilon\|\nabla v\|_{L^2(M)}^2.
\]
Choosing $\epsilon$ small and absorbing the last term yields
\begin{equation}\label{eq:est2v}
\|\nabla v\|_{L^2(M)} \le C\left( \|\nabla k\|_{L^2(M)}+\|v\|_{L^2(K)} \right).
\end{equation}

It remains to remove the compact term. If \eqref{eq:est} were
false, there would be solutions $(k_j,v_j)$ with $L_g^* (v_j) = - \Delta_g k_j$ and 
\[
\|\nabla v_j\|_{L^2(M)}=1, \qquad \|\nabla k_j\|_{L^2(M)}\longrightarrow0.
\]
The Sobolev inequality and local compactness give, after passing to a subsequence, a limit $v\in L^{2^*}(M)$ such that $v_j\to v$ strongly in $L^2(K)$. By \eqref{eq:est2v}, $v\not\equiv0$. On the other hand, $\Delta_gk_j\to0$ distributionally, so passing to the limit in the equation gives $L_g^*v=0$.
The asymptotics of static potentials rule out a nonzero $L^{2^*}(M)$ solution, a contradiction.
\end{proof}

\section{Static vacuum extension and local Bartnik mass minimization}\label{se:ext}

Throughout this section, $(M,\bar g,\bar u)$ is a background asymptotically flat static vacuum triple, possibly with compact boundary $\Sigma$. When $\Sigma=\emptyset$, all boundary conditions on $\Sigma$ are understood to be vacuous. Let $\hat M\subset M$ be an exterior region containing the asymptotically flat end, with smooth compact boundary $\hat\Sigma$, and assume that $\Omega:=\overline{M\setminus\hat M}$ is connected.

To compare deformations on $\hat M$ with those on $M$, we will extend $\hat h\in T_{\bar g}\mathscr C_{\bar g}(\hat M)$ in the geodesic gauge by zero across $\hat\Sigma$. This extension need not be differentiable across $\hat \Sigma$. We first regularize it using Miao's mollification procedure~\cite[Section~3]{Miao:2002}, and then correct the resulting tensors $h_\delta$ so that
\[
h_\delta-k_\delta\in T_{\bar g}\mathscr C_{\bar g}(M).
\]
The key point is that the condition $H'_{\bar g}(\hat h)=0$ prevents a singular contribution to the linearized scalar curvature, leaving an error that tends to zero in $L^p$.

Choose Gaussian coordinates $(s,z)$ near $\hat\Sigma$, with $\bar g=ds^2+\bar\gamma(s)$,  where $s>0$ on $\hat M$ and $s<0$ on $\Omega$. Assume that $\hat h$ satisfies the geodesic gauge $\hat h(\partial_s,\cdot)=0$ along $\hat\Sigma$. Together with $\hat h^\intercal=0$, this gives $\hat h|_{\hat\Sigma}=0$. Its zero extension $h$ is therefore continuous across $\hat\Sigma$. Although its first normal derivative may jump, we have 
\[
	h\in W^{1,p}_{-q}(M) \quad \text{ for every $1\le p < \infty$}.
\]

Fix $\Lambda>8$. Choose nonnegative functions $J,\eta\in C^\infty_c(\mathbb R)$ supported in $[-1,1]$ such that
\[
\int_{\mathbb R}J(t)\,dt=1, \qquad 0\le\eta\le\frac1\Lambda, \qquad \eta=\frac1\Lambda\quad\text{on }[-1/2,1/2].
\]
Set $\eta_\delta(s):=\delta^2\eta(s/\delta)$ and define, componentwise in the Gaussian coordinates,
\[
h_\delta(s,z) := \int_{-1}^1 h\bigl(s-\eta_\delta(s)t,z\bigr)J(t)\,dt.
\]
Extend $h_\delta$ by $h$ outside the coordinate collar. Then $h_\delta=h$ outside $U_\delta:=\{|s|\le\delta\}$. The regularization is $C^2$ across $\hat\Sigma$ and retains the tangential regularity of $h$.

\begin{lemma}[Essentially {\cite[Section~3]{Miao:2002}}]
Let $\hat h\in C^{2,\alpha}_{-q}(\hat M)$ be a symmetric
$(0,2)$-tensor satisfying
\[
R'_{\bar g}(\hat h)=0, \qquad \hat h|_{\hat\Sigma}=0, \qquad H'_{\bar g}(\hat h)=0.
\]
Let $h$ be its zero extension and $h_\delta$ the regularization defined above. Then, for every $1\le p<\infty$,
\begin{equation}\label{eq:mollification-convergence}
\|h_\delta-h\|_{W^{1,p}_{-q}(M)} \longrightarrow0.
\end{equation}
Moreover, there is a constant $C$, independent of $\delta$, such that
\[
R'_{\bar g}(h_\delta)=0 \quad\text{on }M\setminus U_\delta, \qquad |R'_{\bar g}(h_\delta)|\le C \quad\text{on }U_\delta.
\]
Consequently,
\begin{equation}\label{eq:R}
\|R'_{\bar g}(h_\delta)\|_{L^p_{-q-2}(M)} \longrightarrow 0 \qquad\text{for every }1\le p<\infty.
\end{equation}
\end{lemma}
\begin{proof}
Since $h$ is continuous and piecewise $C^{2,\alpha}$, the first derivatives of $h_\delta$ are uniformly bounded. Together with $h_\delta=h$ outside the shrinking collar $U_\delta$, this gives \eqref{eq:mollification-convergence}.

Only the second normal derivatives can be unbounded. On $|s|\le\delta/2$, the mollification scale is constant, $\epsilon_\delta:=\delta^2/\Lambda$, and
\[
\partial_s^2h_\delta(s,z) =[\partial_sh]_0(z)\, \frac1{\epsilon_\delta} J\left(\frac{s}{\epsilon_\delta}\right) +O(1),
\]
where $[\partial_sh]_0 = \partial_s h(0^+) - \partial_s h(0^-)$ denotes the jump across $\hat\Sigma$ and $O(1)$ is uniform in $\delta$. The tangential and mixed second derivatives remain bounded, as do the normal second derivatives outside this inner collar.

In Gaussian coordinates, the only term in $R'_{\bar g}(h_\delta)$ involving two normal derivatives is
\[
-\bar\gamma^{ab}(s,z)\, \partial_s^2(h_\delta)_{ab}.
\]
Since $h=0$ on $\Omega$ and $\hat h|_{\hat\Sigma}=0$, the condition $H'_{\bar g}(\hat h)=0$ gives
\[
\bar\gamma^{ab}(0,z)[\partial_sh_{ab}]_0(z)=0.
\]
Thus the potentially singular term vanishes when its coefficient is evaluated at $s=0$. This proves the uniform bound for $R'_{\bar g}(h_\delta)$.

Outside $U_\delta$, the regularization agrees with $h$, whose linearized scalar curvature vanishes on both sides of $\hat\Sigma$. The uniform bound  therefore implies \eqref{eq:R}.
\end{proof}

\begin{theorem}\label{th:sta}
Let $\hat M\subset M$ be an exterior region and suppose that $ \Omega:=\overline{M\setminus\hat M}$ is connected. If $\bar g$ is a strictly stable critical point of $\mathcal F^M_{\bar u}$ on $\mathscr C_{\bar g}(M)$, then $\bar g$ is a strictly stable critical point of $\mathcal F^{\hat M}_{\bar u}$ on $\mathscr C_{\bar g}(\hat M)$.
\end{theorem}
\begin{proof}
We first transfer the coercive estimate on $M$ to the zero extension of a deformation on $\hat M$, and then compare the gauge projections on the two domains.

Let $\hat h\in T_{\bar g}\mathscr C_{\bar g}(\hat M)$. By diffeomorphism invariance, we may assume that $\hat h$ satisfies the geodesic gauge along  $\hat\Sigma$ (see \cite[Lemma~2.5]{An-Huang:2022}). Thus $\hat h|_{\hat\Sigma}=0$, and its zero extension $h$ to $M$ is continuous and has finite energy. Using superscripts to indicate the domain, Lemma~\ref{le:Q} gives
\begin{align} \label{eq:F}
\left.D^2\mathcal F_{\bar u}^{\hat M}\right|_{\bar g} (\hat h,\hat h)&=Q_{\bar g,\bar u}^{\hat M}(\hat h,\hat h)=Q_{\bar g,\bar u}^{M}(h,h),
\end{align}
where the boundary terms vanish because $\hat h|_{\hat\Sigma}=0$ and $h$ vanishes near $\Sigma$.

To apply strict stability on $M$, we approximate $h$ by deformations satisfying the linearized constraints. Let $h_\delta$ be the regularizations
constructed above. Proposition~\ref{pr:k} provides corrections $k_\delta$ with homogeneous boundary data such that $h_\delta-k_\delta \in T_{\bar g}\mathscr C_{\bar g}(M)$.  The weighted elliptic estimate and \eqref{eq:R} give
\[
\|\nabla k_\delta\|_{L^2(M)}\le C \|k_\delta\|_{W^{2,2}_{-q}(M)} \le C\|R'_{\bar g}(h_\delta)\|_{L^2_{-q-2}(M)} \longrightarrow0.
\]
The continuity of $Q_{\bar g,\bar u}^M$ in the energy norm~\eqref{eq:Qcont} and the gauge estimate \eqref{eq:elliptic} now allow us to
pass strict stability to the limit:
\begin{align}\label{eq:est2}
\begin{split}
Q_{\bar g,\bar u}^M(h,h) &=\lim_{\delta\to0} Q_{\bar g,\bar u}^M(h_\delta - k_\delta, h_\delta- k_\delta)\ge  \liminf_{\delta\to 0} \lambda \| \nabla \big(P^M_{\bar g} (h_\delta- k_\delta)\big) \|^2_{L^2(M)}\\
&\ge \lambda\|\nabla (P_{\bar g}^M h)\|_{L^2(M)}^2.
\end{split}
\end{align}

It remains to relate this lower bound to the gauge norm on $\hat M$. This is not automatic, since admissible diffeomorphisms on $M$ need not fix $\hat\Sigma$. Lemma~\ref{le:proj} below gives the required comparison:
\begin{equation}\label{eq:est3}
\|\nabla \big(P_{\bar g}^{\hat M}\hat h\big)\|_{L^2(\hat M)} \le C_G\|\nabla \big( P_{\bar g}^M h\big) \|_{L^2(M)},
\end{equation}
where $C_G$ is independent of $\hat h$. Combining \eqref{eq:F},  \eqref{eq:est2}, and \eqref{eq:est3}, we obtain
\[
\left.D^2\mathcal F_{\bar u}^{\hat M}\right|_{\bar g} (\hat h,\hat h) \ge
\lambda C_G^{-2} \|\nabla \big( P_{\bar g}^{\hat M}\hat h\big)\|_{L^2(\hat M)}^2.
\]
This proves strict stability on $\hat M$.
\end{proof}

We establish the gauge comparison \eqref{eq:est3} in the following two lemmas.
\begin{lemma}\label{le:Korn}
Let $M,\hat M$ be as in Theorem~\ref{th:sta}, and suppose that $\Omega:=\overline{M\setminus\hat M}$ is connected. There exists a constant $C>0$ such that, for every vector field $X\in H^2(\Omega)$ satisfying $X|_\Sigma=0$ when $\Sigma\neq\emptyset$, there is a Killing vector field $Z$ on $M$ with
\begin{equation}\label{eq:korn-core}
\|X-Z\|_{H^2(\Omega)} \le C\|\mathcal L_X\bar g\|_{H^1(\Omega)}.
\end{equation}
Here $Z=0$ when $\Sigma\neq\emptyset$. In the boundaryless case,  $(M,\bar g)$ is Euclidean, and $Z$ is a Euclidean Killing vector field.
\end{lemma}

\begin{proof}
Set $Z=0$ when $\Sigma\neq\emptyset$; otherwise, let $Z$ be the $L^2(\Omega)$-orthogonal projection of $X$ onto the restrictions of Euclidean Killing fields. Write $V=X-Z$.

By Korn's inequality~\cite[Theorem~1.2]{Chen-Jost:2002}, 
\begin{equation}\label{eq:korn-inhomogeneous}
\|V\|_{H^1(\Omega)} \le C\left( \|V\|_{L^2(\Omega)} +\|\mathcal L_V\bar g\|_{L^2(\Omega)} \right).
\end{equation}
A standard compactness argument therefore removes the $L^2$-term in \eqref{eq:korn-inhomogeneous}, giving
\begin{equation}\label{eq:korn-H1}
\|V\|_{H^1(\Omega)} \le C\|\mathcal L_V\bar g\|_{L^2(\Omega)}.
\end{equation}

The identity for second covariant derivatives of a vector field in terms of its Killing operator gives
\[
\|\nabla^2V\|_{L^2(\Omega)} \le C\left( \|\nabla(\mathcal L_V\bar g)\|_{L^2(\Omega)} +\|V\|_{L^2(\Omega)} \right).
\]
Combining this with \eqref{eq:korn-H1} and $\mathcal L_V\bar g=\mathcal L_X\bar g$ proves \eqref{eq:korn-core}.
\end{proof}

The preceding estimate will allow us to modify a global gauge field so that it vanishes on $\hat\Sigma$.
\begin{lemma}\label{le:proj}
Let $M,\hat M$ be as in Lemma~\ref{le:Korn}. There exists a constant $C_G>0$, depending only on $(M,\hat M,\bar g)$ and the two gauge projections, such that
\[
\|\nabla \big( P_{\bar g}^{\hat M}\hat h\big) \|_{L^2(\hat M)} \le C_G\|\nabla \big( P_{\bar g}^{M}h\big)\|_{L^2(M)}
\]
for every symmetric $(0,2)$-tensor $\hat h\in C^{2,\alpha}_{-q}(\hat M)$ with $\hat h|_{\hat \Sigma}=0$ and the corresponding zero extension $h$ on $M$.
\end{lemma}

\begin{proof}
By Definition~\ref{de:gauge},  write
\begin{equation}\label{eq:v-global-gauge}
h_0:=P_{\bar g}^{M}h=h+\mathcal L_X\bar g, \qquad X\in\mathcal X^{2,2}(M).
\end{equation}
Although $X$ need not vanish on $\hat\Sigma$, the identity $h=0$ on $\Omega$ gives
\[
h_0=\mathcal L_X\bar g \qquad\text{on }\Omega.
\]
Thus Lemma~\ref{le:Korn}  provides a global Killing field $Z$ such that
\[
\|X-Z\|_{H^2(\Omega)} \le C\|h_0\|_{H^1(\Omega)} \le C\|\nabla h_0\|_{L^2(M)}
\]
where we use the weighted Poincar\'e inequality in the last inequality. 

The trace theorem and a bounded extension operator give a vector field $W$ on $\hat M$, supported in a fixed collar of $\hat\Sigma$, with
\[
W|_{\hat\Sigma}=(X-Z)|_{\hat\Sigma}
\]
and
\begin{equation}\label{eq:W-estimate}
\|W\|_{H^2(\hat M)} \le C\|(X-Z)|_{\hat\Sigma}\|_{H^{3/2}(\hat\Sigma)} \le C\|\nabla h_0\|_{L^2(M)}.
\end{equation}

The field
\[
Y:=(X-Z)|_{\hat M}-W
\]
vanishes on $\hat\Sigma$ and has the admissible asymptotics, so $Y\in\mathcal X(\hat M)$. Since $Z$ is globally Killing, \eqref{eq:v-global-gauge} gives
\[
\hat h+\mathcal L_Y\bar g =h_0|_{\hat M}-\mathcal L_W\bar g \quad \mbox{ on } \hat M.
\]
The gauge projection is unchanged by adding $\mathcal L_Y\bar g$. Hence \eqref{eq:elliptic} and \eqref{eq:W-estimate} yield
\begin{align*}
\|\nabla (P_{\bar g}^{\hat M}\hat h)\|_{L^2(\hat M)} &\le C\|\nabla(h_0|_{\hat M}-\mathcal L_W\bar g)\|_{L^2(\hat M)}\\
&\le C\left( \|\nabla h_0\|_{L^2(M)}+\|W\|_{H^2(\hat M)} \right)\\
&\le C_G\|\nabla h_0\|_{L^2(M)}.
\end{align*}
This is the desired estimate.
\end{proof}

We now use strict stability to verify the nondegeneracy condition in the local static vacuum extension theorem in our previous work~\cite{An-Huang:2024}. We recall this theorem on a general exterior region $N$ before applying it to $\hat M$.

Define
\[
\mathcal T(g,u) :=\left( -u\Ric_g+\nabla_g^2u,\, \Delta_gu,\, g^\intercal,\, H_g \right).
\]
The first two components are defined on $N$, while the last two are defined on $\partial N$. Thus a static vacuum extension of the Bartnik boundary data
$(\tau,\phi)$ satisfies
\[
\mathcal T(g,u)=(0,0,\tau,\phi).
\]
At a static vacuum pair $(\bar g,\bar u)$, the linearization is 
\begin{equation}\label{eq:lin-sta}
\left.D\mathcal T\right|_{(\bar g,\bar u)}(h,v) = \left( \begin{array}{c} -\bar u\,\Ric'_{\bar g}(h) +(\nabla^2)'_{\bar g}(h)\bar u -v\,\Ric_{\bar g} +\nabla_{\bar g}^2v
\\[1mm]
\Delta_{\bar g}v+\Delta'_{\bar g}(h)\bar u
\\[1mm]
h^\intercal
\\[1mm]
H'_{\bar g}(h)
\end{array}
\right).
\end{equation}
Diffeomorphism invariance implies that its kernel contains the pure-gauge variations
\[
\bigl(\mathcal L_X\bar g,X(\bar u)\bigr), \qquad X\in\mathcal X(N).
\]
The following theorem reduces local existence to showing that these are the only kernel elements. This theorem is the only place where we need to assume that the boundary $\partial N$ is connected.

\begin{theorem}[{\cite[Theorem~3]{An-Huang:2024}}]
\label{thm:existence}
Let $(N,\bar g,\bar u)$ be an asymptotically flat static vacuum triple with connected compact boundary $\partial N$. Suppose that $\bar u>0$ and
\[
\ker\left.D\mathcal T\right|_{(\bar g,\bar u)} =\left\{
\bigl(\mathcal L_X\bar g,X(\bar u)\bigr):
X\in\mathcal X(N)
\right\}.
\]
Then there exist constants $\epsilon,C>0$ such that, for every pair of 
boundary data $(\tau,\phi)$ satisfying
\[
\|\tau-\bar g^\intercal\|_{C^{2,\alpha}(\partial N)}
+
\|\phi-H_{\bar g}\|_{C^{1,\alpha}(\partial N)}
<\epsilon,
\]
there is an asymptotically flat static vacuum pair $(g,u)$ with
\[
(g^\intercal,H_g)=(\tau,\phi)
\quad\text{on }\partial N
\]
and
\[
\begin{split}
&\|g-\bar g\|_{C^{2,\alpha}_{-q}(N)}
+\|u-\bar u\|_{C^{2,\alpha}_{-q}(N)}\\
&\qquad\le
C\left(
\|\tau-\bar g^\intercal\|_{C^{2,\alpha}(\partial N)}
+
\|\phi-H_{\bar g}\|_{C^{1,\alpha}(\partial N)}
\right).
\end{split}
\]
In an elliptic gauge, the solution depends smoothly on the
boundary data. For fixed boundary data, it is locally unique
up to the action of $\mathscr D(N)$.
\end{theorem}

The link with strict stability is that every kernel element
determines a null direction of the constrained Hessian.

\begin{lemma}\label{lem:nondeg}
Let $(N,\bar g,\bar u)$ be an asymptotically flat static vacuum
triple with $\bar u>0$. If $(h,v)\in
\ker\left.D\mathcal T\right|_{(\bar g,\bar u)}$,  then $h\in T_{\bar g}\mathscr C_{\bar g}(N)$ and $\left.D^2\mathcal F_{\bar u}\right|_{\bar g}(h,h)=0$. 
\end{lemma}

\begin{proof}
Taking the trace of the first equation in \eqref{eq:lin-sta}
and using the second gives $\bar uR'_{\bar g}(h)=0$.
Since $\bar u>0$, the linearized scalar curvature vanishes.
Together with the boundary equations, this shows that
$h\in T_{\bar g}\mathscr C_{\bar g}(N)$.

Substituting the linearized static equations into the second
variation formula and integrating by parts, we obtain
\begin{align*}
D^2\mathcal F_{\bar u}\big|_{\bar g}(h,h)&=\int_N \Big\langle \bar u\,\operatorname{Ric}'_{\bar g}(h) -(\nabla^2)'_{\bar g}(h)\bar u +\bigl(\Delta'_{\bar g}(h)\bar u\bigr)\bar g, h \Big\rangle_{\bar g} \,d\mu_{\bar g} \\
&= \int_N \Big\langle -v\operatorname{Ric}_{\bar g} +\nabla_{\bar g}^2v -(\Delta_{\bar g}v)\bar g, h \Big\rangle_{\bar g} \,d\mu_{\bar g} \\
&= \int_N \left\langle L_{\bar g}^*v,h \right\rangle_{\bar g} \,d\mu_{\bar g}=\int_N v\,R'_{\bar g}(h)\,d\mu_{\bar g}=0. 
\end{align*}
 The boundary terms vanish by $h^\intercal=0$,
$H'_{\bar g}(h)=0$, and the prescribed decay;
see \eqref{eq:ibp2}.
\end{proof}
\begin{proof}[Proof of Theorem~\ref{Th:extension}]
By Theorem~\ref{th:sta}, strict stability on $M$ passes to
$\hat M$. It remains to verify the kernel condition in
Theorem~\ref{thm:existence}.

Let $(h,v)\in
\ker\left.D\mathcal T\right|_{(\bar g,\bar u)}$ on $\hat M$.
Lemma~\ref{lem:nondeg} and strict stability give $\nabla (P_{\bar g}^{\hat M}h)=0$. Since $P_{\bar g}^{\hat M}h$ decays, it vanishes identically.
Thus $h=\mathcal L_X\bar g$ for some $X\in\mathcal X(\hat M)$.
Subtracting the corresponding pure-gauge kernel element gives
\[
(0,w)\in
\ker\left.D\mathcal T\right|_{(\bar g,\bar u)},
\qquad
w:=v-X(\bar u).
\]
The linearized equations imply $L_{\bar g}^*w=0$. Since $w$ decays, it  implies $w=0$.

Hence the kernel consists precisely of pure-gauge variations.
Theorem~\ref{thm:existence} supplies the desired static vacuum
extensions, and Proposition~\ref{pr:str} ensures strict stability. 
\end{proof}

For Theorem~\ref{Th:local}, we first deform the metric to be 
scalar flat without increasing its mass. The resulting mean
curvature need not equal that of $\bar g$, so we use static vacuum extensions to bridge this difference.

\begin{proof}[Proof of Theorem~\ref{Th:local}]
Let $\varphi$ be the unique positive solution of
\[
-\frac{4(n-1)}{n-2}\Delta_g\varphi+R_g\varphi=0,
\qquad
\varphi|_\Sigma=1,
\qquad
\varphi\longrightarrow1
\quad\text{at infinity},
\]
and set $\tilde g:=\varphi^{4/(n-2)}g$.
Since $R_g\ge0$, the maximum principle gives
$0<\varphi\le1$ and $\nu_g(\varphi)\le0$ on $\Sigma$.
The conformal transformation formulas therefore give $R_{\tilde g}=0$,  $\tilde g^\intercal=\bar g^\intercal$,  and 
\[
H_{\tilde g}= H_g+\frac{2(n-1)}{n-2}\nu_g(\varphi) \le H_g\le H_{\bar g}.
\]
Moreover, the expansion
$\varphi=1+Ar^{2-n}+o(r^{2-n})$ has $A\le0$, so
\[
m_{\ADM}(\tilde g) = m_{\ADM}(g)+2A \le m_{\ADM}(g).
\]

We next compare the boundary data of $\tilde g$ and $\bar g$.
Write $\bar\gamma=\bar g^\intercal$ and set
\[
H_s:=(1-s)H_{\bar g}+sH_{\tilde g}, \qquad 0\le s\le1.
\]
Theorem~\ref{Th:extension} gives a smooth family of static
vacuum pairs $(g_s,u_s)$ satisfying
\[
(g_s^\intercal,H_{g_s})=(\bar\gamma,H_s), \qquad (g_0,u_0)=(\bar g,\bar u).
\]
After shrinking the original neighborhood, the potentials
$u_s$ are positive, the metrics $g_s$ are strictly stable,
and the local minimizing estimate holds at $g_1$. Since $\tilde g$ and $g_1$ are scalar flat
and have the same Bartnik boundary data,
Theorem~\ref{Th:minimizer} gives
\[
m_{\ADM}(\tilde g)\ge m_{\ADM}(g_1).
\]

Along the static family, the induced boundary metric is fixed,
so the first variation formula reduces to
\[
\frac{d}{ds}m_{\ADM}(g_s) = -\frac{1}{(n-1)\omega_{n-1}} \int_\Sigma u_s(H_{\tilde g}-H_{\bar g})\,d\sigma_{\bar\gamma} \ge 0.
\]
Consequently,
\[
m_{\ADM}(g) \ge m_{\ADM}(\tilde g) \ge m_{\ADM}(g_1) \ge m_{\ADM}(\bar g).
\]

If equality holds, the mass variation formula and positivity
of $u_s$ imply $H_{\tilde g}=H_{\bar g}$.
The conformal mean curvature formula then gives
$H_g=H_{\bar g}$ and $\nu_g(\varphi)=0$ on $\Sigma$.
The boundary point lemma forces $\varphi\equiv1$, so $R_g=0$.
The rigidity statement of Theorem~\ref{Th:minimizer} now
completes the proof.
\end{proof}
\bibliographystyle{amsplain}
\bibliography{2026stability}
\end{document}